%% file: arXiv_submission.tex
\documentclass[a4paper,reqno]{amsart}

\pdfoutput=1

\usepackage{amscd,amsthm,amsmath,amssymb,mathtools}
\usepackage{upgreek,bm}
\usepackage{verbatim}
\usepackage{color}
\usepackage{hyperref}
\usepackage{url}
\usepackage{enumerate,enumitem}

\usepackage{framed}

\usepackage{graphicx}
\usepackage{subcaption}
\usepackage{epstopdf,epsfig}
\usepackage{curve2e}

\usepackage{array}

\usepackage{algorithm}
\usepackage{algorithmic}

\usepackage{tikz}
\usetikzlibrary{calc,positioning,arrows.meta}
\tikzstyle{every picture}+=[remember picture,overlay]
\usepackage{xcolor}
\usepackage{pgfplots}
\pgfplotsset{compat=1.18}
\usepackage{tikz}
\usetikzlibrary{decorations.pathmorphing}

\usepackage[numbers,sort&compress]{natbib}

\usepackage{setspace}
\usepackage[nameinlink,noabbrev]{cleveref}
\crefname{subsection}{subsection}{subsections}
\Crefname{subsection}{Subsection}{Subsections}

\theoremstyle{plain}

\newtheorem{theorem}{Theorem}[section]

\newtheorem{corollary}[theorem]{Corollary}

\theoremstyle{definition}
\newtheorem{definition}[theorem]{Definition}
\newtheorem{remark}[theorem]{Remark}

\numberwithin{equation}{section}

\usepackage{geometry}

\input{Mathcommands}

\newcommand{\bseta}{{\boldsymbol{\eta}}}
\newcommand{\bsGammas}{{\boldsymbol{\Gamma}_{s+1}}}
\newcommand{\bsz}{{\boldsymbol{z}}}
\newcommand{\bsy}{{\boldsymbol{y}}}

\usepackage{xargs}

\begin{document}
\title{Dynamic output-feedback stabilization of uncertain linear dynamics via digital twins}
\author{Philipp A.~Guth$^{\tt1}$,  Karl Kunisch$^{\tt 1,
2}$, S\'ergio S.~Rodrigues$^{\tt3}$, and Jesper Schr\"oder$^{\tt1}$}
\thanks{
\vspace{-1em}\newline\noindent
\newline\noindent
$^{\tt1}$ Johann Radon Institute for Computational and Applied Mathematics,
  \"OAW, %
  Altenbergerstrasse~69, 4040~Linz, Austria.\newline\noindent
$^{\tt2}$   Institute of Mathematics and Scientific Computing,    Karl-Franzens University of Graz,	     	
Heinrichstrasse~36, 8010 Graz, Austria, and Johann Radon Institute for Computational and Applied Mathematics,
  \"OAW, %
  Altenbergerstrasse~69, 4040~Linz, Austria.\newline\noindent
  $^{\tt3}$ Departament of Mathematics,
FCT, NOVA University of Lisbon,
2829-516 Caparica,
Portugal
\newline\noindent
{\sc Emails}:
{\small\tt   philipp.guth@ricam.oeaw.ac.at,\quad karl.kunisch@uni-graz.at,\quad\\ \hspace*{3.4em}ssi.rodrigues@fct.unl.pt,\quad jesper.schroeder@ricam.oeaw.ac.at}
 }

\begin{abstract}
This work presents a digital twin framework for output-feedback stabilization and parameter identification in uncertain dynamical systems. A virtual model evolves in parallel with the physical process, assimilating measurement data in real time. By design, the digital twin reconstructs the system state and generates a stabilizing feedback, while model parameters are simultaneously inferred from data of the controlled dynamics using a Bayesian approach. Numerical results for the coupled physical-virtual dynamics demonstrate how digital twins can act jointly as observers, parameter estimators, and control agents, ensuring robust performance under uncertainty.
\end{abstract}

\maketitle

\pagestyle{myheadings} \thispagestyle{plain} \markboth{\sc P.A. Guth,  K. Kunisch, S.S. Rodrigues, J. Schr\"oder}
{\sc Dynamic output-feedback stabilization of uncertain linear dynamics}



\section{Introduction}
Digital twins (DT) have become a useful tool in  a multitude of 
applied sciences, including engineering, healthcare, supply chains, and environmental systems, for example. While important and practical  experience has been obtained and documented, the mathematical analysis  poses significant challenges to be overcome. In this research we take a step in this direction. To commence, let us specify the notion of DT that we shall follow. 

\smallskip
\begin{center}\parbox{0.9\textwidth}{%
  \itshape
  A DT is a coupled system consisting  of  a set of virtual information constructs, as for instance an underdetermined dynamical system, that describes the structure and behavior of the physical system, 
from which it is updated by data as time progresses. The DT has the capability to influence the physical system to achieve predefined objectives. A  bidirectional interaction between the virtual twin  and the physical twin  is central to the digital twin.
}
\smallskip 

\end{center}
This is a slightly modified description of a DT guided by the definition given by the National Academies of Science, Engineering, and Medicine \cite{NASEM2024}, which itself is a slightly altered and extended form of an earlier definition by the American Institute of Aeronautics and Astronautics \cite{aiaa2020digitaltwin}. For an interesting introduction to mathematical aspects of DTs we also refer to \cite{antil2024}.

For the successful analysis of DTs, many techniques of mathematical systems theory come into play. These include parameter estimation, inverse problems, uncertainty quantification, feedback and possibly optimal feedback control, state estimation, and reduced order modeling. Due to the coupling between the virtual and the physical twin, these methods must be employed in an interconnected manner, which implies a need for the development of new concepts. We attempt to go a step in this direction. There are many other important applied and numerical  challenges which arise in the context of DTs, including large data sets, the curse of dimensionality, and multiscale aspects. These topics are not part of the present work.

In the development of digital twins, the nature of the underlying problem typically dictates whether rich data sets or highly-accurate physical models are available. Rarely, both are available simultaneously. While data-driven approaches leverage large datasets to infer system behavior, model-based strategies capitalize on detailed knowledge of the underlying physics and therefore remain reliable even when measurements are sparse, noisy, or only indirectly related to quantities of interest. In this work, we operate on the model-driven end of the spectrum, where the governing equations provide the structural backbone of the digital twin and the available measurements are used to calibrate the model within a holistic framework for prediction and control that allows for uncertainty quantification  and ultimately leads to guaranteed stabilization of the physical system.

In this manuscript, the main purpose of the design of the DT is oriented towards the construction of controls stabilizing the physical system. In the context of uncertain dynamics, this  requires the construction of the virtual twin to take into account the nontrivial task of online identification of appropriate estimates for unknown parameters of the dynamics. Since the data obtained from the physical twin only partially describe its state,  state reconstruction has to be taken into consideration as well. 

The literature on controlling dynamical systems under parametric uncertainty is as rich as it is diverse. The following overview does not claim completeness and is to be understood as a starting point for the interested reader. A broad overview from a system theoretic standpoint at the turn of the century is presented in \cite{AstWit95}. This includes, in particular, \textit{self-tuning regulators} \cite[Ch.~3 and 4]{AstWit95}, based on the principle of estimating uncertain parameters and adjusting the controller as if these parameters were known. 
This approach, which is referred to as the \textit{certainty equivalence principle}, does not take into account the uncertainties of the estimated parameters. For more recent accounts of adaptive control, we refer to \cite{AnnaswamyFradkov2021, LevanonyCaines}, for example. A conceptually different strategy is given by a stochastic formulation where controls are chosen according to a cost functional including a control goal and the parameter uncertainty. This leads to the concept of \textit{dual control} \cite[Ch.~7]{AstWit95}, where optimal controls strike a balance between achieving the desired behavior and properly identifying parameters. We refer to \cite{Mes18} for a survey on dual control and to \cite{VenKoBeAl25}
proposing an explicit construction of controls sequentially exploring parameters and controlling the system towards a desired behavior.
In \cite{DeyDhBh23} the authors present an MPC-based strategy combining control, state estimation, and parameter estimation.
Another approach is given by \textit{gain scheduling} \cite[Ch.~9]{AstWit95}, where controllers are adapted according to measurements of the operating environment, see also \cite{VenKoBeAl20} for a more recent reference relying on this concept. 
Finally, we mention \cite{StrKaWi18} presenting a fully Bayesian approach utilizing the Ensemble Kalman filter.

This work presents a novel design for simultaneous parameter identification and system stabilization. Given flawed information on the initial state and partial, noisy measurements of the evolving state of the physical twin, the virtual twin is updated iteratively using Bayesian inference. This identification procedure is carried out in conjunction with feedback control and state estimation which are performed continuously using Riccati gains based on the current parameter estimate of the virtual twin.

The paper is structured by the following sections. \Cref{sec:design_DT} provides a detailed description of our  problem setting. The asymptotic behavior of the coupled  physical-virtual twin system without parameter updates is investigated in \Cref{sec:asympt}.  The Bayesian update strategy for the parameters of the virtual twin is described in \Cref{sec:param}. \Cref{sec:Kalmanfilter} explains the Kalman-filter estimation strategy for the coupled twin state.  The final \Cref{sec:numerics} is devoted to the investigation of the proposed strategies in numerical practice. 
\section{Design of the digital twin}\label{sec:design_DT}
Suppose a real-world phenomenon evolves, for time~$t>0$, according to
\begin{equation}\label{sys-intro}
\tfrac{\rmd}{\rmd t} y(t)=A_\sigma y(t) +Bu(t), \;\text{ in } \mathbb{R}^n,
\end{equation}
with unknown initial condition $y(0) = y_0$. Motivated by the fact that modeling errors are ubiquitous in applications, we consider an uncertain parameter~$\sigma\in\bbR^p$ influencing the dynamics via~$A_\sigma$. The linear operator~$A_\sigma$ defining the free dynamics is independent of time and depends continuously on $\sigma$.  The control input is~$u(t)\in\bbR^{m}$, and the linear control operator~$B\colon\bbR^m\to \mathbb{R}^n$ is independent of time. Furthermore, we collect noisy output data according to
\begin{align}\label{eq:obs}
z(t_j) = C y(t_j) + \eta_j, \quad \text{where} \quad \eta_j \sim \mathcal{N}(0,\Gamma_j),\quad t_j\in (t_0,\ldots,t_s),
\end{align}
where the noise~$\eta_j$, $j =0,\ldots,s$ is independently normally distributed with symmetric positive definite covariance matrix $\Gamma_j$. At every measurement point $t_j$, the vector~$z$ represents the output of $\ell$~sensor measurements, with linear output operator~$C\colon \bbR^{n } \to\bbR^{\ell }$.
Above,~$n$, $m$, and~$\ell$ are fixed positive integers.

Our goal is to find an input feedback operator~$K_{\widehat{\sigma}}$ and take the input~$u(t)=K_{\widehat{\sigma}} \widehat y(t)$, where~$\widehat y(t)$ is an estimate for~$y(t)$ and ${\widehat{\sigma}}$ is an estimate for $\sigma$, which is updated as time progresses, so that the solution of the real-world system~\eqref{sys-intro}
\begin{equation}\label{eq:sys-cltwin}
\dot y(t)=A_\sigma y(t) +BK_{\widehat{\sigma}} \widehat y(t),
\end{equation}
converges to zero as time increases, that is we want to stabilize the physical system~\eqref{sys-intro}. For brevity, we omit the explicit dependence on $t$ whenever it is clear from the context.%

The input~$u=K_{\widehat{\sigma}}\widehat{y}$ is obtained based on a virtual model that sequentially assimilates data and mirrors the dynamics of the physical system. This model evolves simultaneously with the physical system and is governed by
\begin{align}\label{eq:digitaltwin}
    &\dot {\widehat y}=A_{\widehat{\sigma}} \widehat y +BK_{\widehat{\sigma}}\widehat y+L_{\widehat{\sigma}} C(\widehat y-y) - L_{\widehat{\sigma}}\eta,\quad &&\widehat y(0)\coloneqq \widehat y_0;
\end{align}
Designing the virtual twin in \eqref{eq:digitaltwin} entails addressing several interrelated challenges.
\begin{itemize}
    \item Estimating~$\widehat{\sigma}$ based on the output measurement data collected from the physical system \eqref{eq:obs}. This involves solving a sequence of inverse problems, which are typically ill-posed and sensitive to noise, while ensuring that the estimated parameters produce a dynamically consistent model suitable for reliable controller synthesis and observer design.
    \item Constructing the state-feedback operator~$K_{\widehat{\sigma}}$, a task commonly referred to as \emph{controller synthesis} or \emph{state-feedback design}. The goal is to shape the system's closed-loop behavior in accordance with performance criteria, which in this manuscript is stabilization. In the presence of uncertainty in the parameter estimate~$\widehat{\sigma}$, this becomes a problem of \emph{robust} or \emph{risk-aware} control.
    \item Designing the observer gain~$L_{\widehat{\sigma}}$ to ensure stable and accurate reconstruction of the system state from noisy and possibly partial measurements. This task is often referred to as \emph{observer design} or \emph{continuous data assimilation}. It involves balancing responsiveness to new data with robustness to measurement noise. It plays a crucial role in ensuring that the virtual state~$\widehat{y}$ remains synchronized with the true physical state~$y$.
\end{itemize}
The main contribution of this manuscript is a mathematical framework that jointly addresses the interacting challenges identified above, with the goal of stabilization of the physical system.

\subsection{The case without uncertain parameter and without measurement noise}
If we knew~$\sigma\in\bbR^p$ and if we had exact observations, i.e., $z = Cy$, a classical strategy is to seek a feedback-input operator~$K_\sigma$ and an output-injection operator~$L_\sigma$ so that
both operators~$A_\sigma +BK_\sigma$ and~$A_\sigma +L_\sigma C$ are exponentially stable. Once we find these operators, it follows that the coupled closed-loop system%
\begin{subequations}\label{digital-twin-intro}
\begin{align}
&\dot y=A_\sigma y +BK_\sigma\widehat y,\quad &&y(0)\coloneqq y_0;\\
&\dot {\widehat y}=A_\sigma \widehat y +BK_\sigma\widehat y+L_\sigma C(\widehat y-y),\quad &&\widehat y(0)\coloneqq \widehat y_0;
\end{align}
\end{subequations}
provides us with the sought stabilizing input~$K_\sigma\widehat y$. Here,~$\widehat y_0$ can be taken as an initial guess that we might have for~$y_0$.

As we can see, the estimate~${\widehat y}$ is given by a dynamic Luenberger observer, namely, by a dynamical system consisting of a copy of the physical system plus a correction term given by the injected forcing~$L_\sigma C(\widehat y-y)=L_\sigma (C\widehat y-z)$.

\subsection{The case with an uncertain parameter and measurement noise}
If the parameter~$\sigma$ is unknown, the classical strategy as described above is not directly applicable, since the operators~$K_\sigma$ and~$L_\sigma$ in~\eqref{digital-twin-intro} rely on exact knowledge of $\sigma$. Furthermore, measurement noise in the output propagates through~\eqref{digital-twin-intro}. Even if the parameter $\sigma$ was known, this persistent noise generally prevents the state from converging to zero almost surely. To overcome these limitations, we propose an adaptive strategy that incorporates online parameter identification. Specifically, we compute an estimate~$\widehat{\sigma}$ of the true parameter~$\sigma$ based on the system model \eqref{sys-intro} and the observation process \eqref{eq:obs}. Using Bayes' law, our prior belief about the parameter is continuously updated as new measurement data become available. Given an estimate~$\widehat{\sigma}$, we can mimic~\eqref{digital-twin-intro}
\begin{subequations}\label{digital-twin-intro-unc}
\begin{align}
&\dot y=A_\sigma y +BK_{\widehat{\sigma}}\widehat y,\quad &&y(0)\coloneqq y_0;\\
&\dot {\widehat y}=A_{\widehat{\sigma}} \widehat y +BK_{\widehat{\sigma}}\widehat y+L_{\widehat{\sigma}} C(\widehat y-y) - L_{\widehat{\sigma}}\eta,\quad &&\widehat y(0)\coloneqq \widehat y_0.\label{eq:digital-twin}
\end{align}
\end{subequations}
Since the observer uses~$A_{\widehat{\sigma}}$ instead of~$A_{\sigma}$, it is not guaranteed that the input~$K_{\widehat{\sigma}}\widehat y$ will be able to stabilize~$y$. For this reason, it is necessary to update~${\widehat{\sigma}}$ online so that~$A_{\widehat{\sigma}}$ progressively approaches~$A_{\sigma}$. In doing so, the dynamics of \eqref{digital-twin-intro-unc} better approximate that of \eqref{digital-twin-intro}, up to the unavoidable influence of measurement noise. We shall compute a piecewise constant function~$\widehat\sigma$ by updating~$\widehat\sigma$ periodically over time. As feedback gain $K_{\widehat{\sigma}}$ and observer gain $L_{\widehat{\sigma}}$ we shall utilize the solutions to the appropriate algebraic Riccati equations.
\subsection{Our digital twin strategy}
In order to find the sought stabilizing control input in the presence of an uncertain parameter $\sigma$ and incomplete noisy data $Cy+\eta$, we design an observer that operates in parallel with the physical system. It synchronizes with the physical system through output data collected according to \eqref{eq:obs}, computes estimates $\widehat{y}$ of $y$ using Kalman filtering, and $\widehat{\sigma}$ of $\sigma$ using Bayesian estimation. Subsequently, the digital twin acts on the physical system via the input $u(t) = K_{\widehat{\sigma}(t)}\widehat y(t)$. In this sense, the observer acts as a \emph{virtual twin} for the \emph{physical twin} \eqref{sys-intro}. Importantly, the approach illustrated in \Cref{figure1} does not rely on numerical evaluations or direct knowledge of the physical system with the true parameter. Instead, the physical system is treated as an existing entity that receives the feedback input $u$ and supplies the measurements~$z$.
\begin{figure}[htb]
    \input{figure1}
\caption{Illustration of the Digital Twin.}
\label{figure1}
\end{figure}
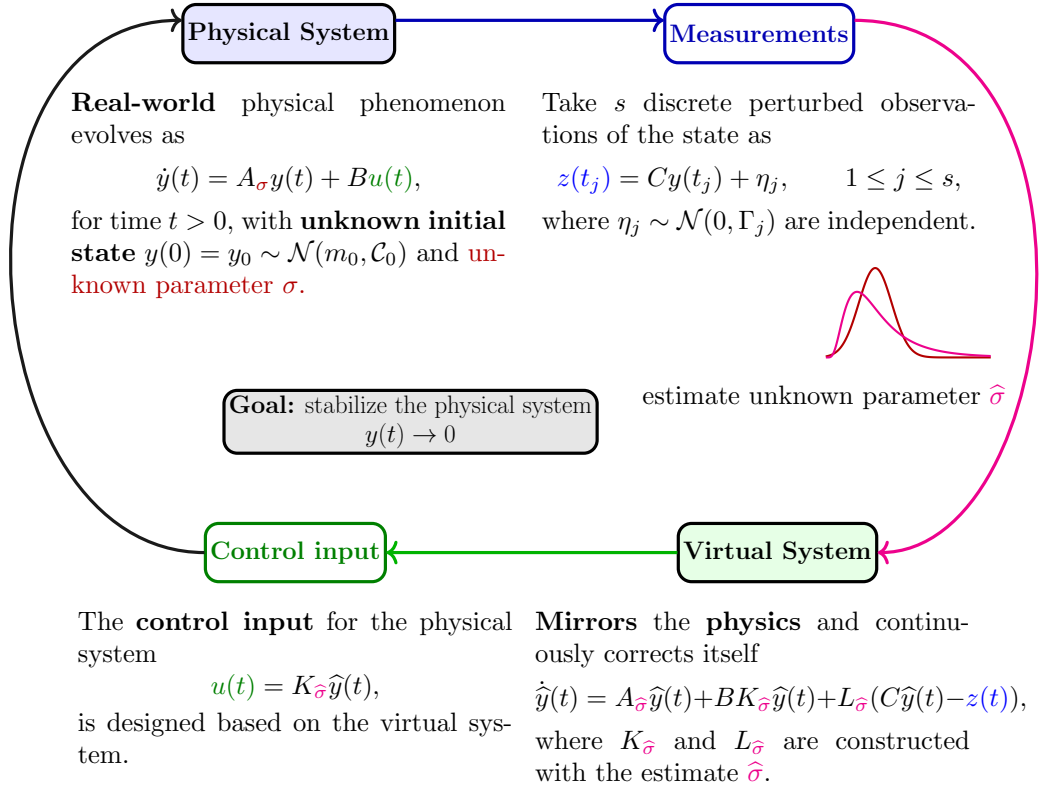

The reader may at this point already consult \Cref{Alg:Offl} in \Cref{sec:numerics}.

\subsection{Mathematical formulation of the virtual twin}

Output measurements of the physical system are available only at discrete time instances, and the virtual twin itself is implemented as a virtual computational model. It is therefore natural to formulate the virtual twin as a discrete-time dynamical system. To make this precise, and consistent with \eqref{eq:obs}, we assume that, between two consecutive parameter updates, with current estimate $\widehat{\sigma}$, there are~$s$ discrete observation times, see \Cref{Alg:Offl}. For~$t \in [t_j,t_{j+1}),\, j =0,\ldots,s-1$ the controlled physical system is given as
\begin{align}\label{eq:y_cont}
    \dot y(t)&=A_\sigma y(t) +BK_{\widehat{\sigma}}\widehat y_j,&&y(t_j)\coloneqq y_j,
\end{align}
where the virtual twin input $\widehat y_j(t) = \widehat y_j$, for all $t\in [t_j,t_{j+1})$ is constant between the observation points. Let us assume an equidistant grid $t_j = j \Delta t$, $\forall j \ge 0$. The exact solution of~\eqref{eq:y_cont} evaluated on the interval boundaries gives the iteration map 
\begin{align}
    y_{j+1} 
    &= e^{A_\sigma \Delta t} \, y_j + \int_0^{\Delta t} e^{A_\sigma(\Delta t-\tau)}\,\mathrm d\tau\, B K_{\widehat{\sigma}} \,\widehat y_j. \label{eq:y_iter}
\end{align}

The dynamics~\eqref{eq:y_iter} can be expressed as
\begin{align}\label{eq:y_discr}
    y_{j+1} = \clA_\sigma y_j + \clB_{\sigma} K_{\widehat{\sigma}} \widehat{y}_j,
\end{align}
for $\clA_\sigma = e^{A_\sigma \Delta t}$ and $\clB_{\sigma} = \int_0^{\Delta t} e^{A_\sigma(\Delta t-\tau)}\,\mathrm d\tau\, B$ with initializations $y_0$ and $\widehat{y}_0$ as in~\eqref{digital-twin-intro}.
Based on~$\clA_\sigma$,~$\clB_{\sigma}$, and an estimate~$\widehat{\sigma}$ of $\sigma$, we design the virtual twin as
\begin{align*}
    \widehat y_{j+1} = \clA_{\widehat{\sigma}} \widehat y_j + \clB_{\widehat{\sigma}} K_{\widehat{\sigma}} \widehat{y}_j + L_{\widehat{\sigma}}(C \widehat{y}_j - z_j),\quad \widehat{y}_0 \text{ is known}.
\end{align*}
In this way, we arrive at the coupled discrete-time system 
\begin{subequations}\label{eq:timeDiscSys_diffsigma}
    \begin{align}
        y_{j+1} 
        &=
        \clA_\sigma y_j + \clB_\sigma K_{\widehat{\sigma}} \widehat{y}_j,\label{eq:timeDiscSys_diffsigma1}\\
        z_j 
        &=
        C y_j + \eta_j,\\
        \widehat{y}_{j+1} 
        &=
        \clA_{\widehat{\sigma}} \widehat{y}_j 
        + \clB_{\widehat{\sigma}} K_{\widehat{\sigma}} \widehat{y}_j
        + L_{\widehat{\sigma}} (C \widehat{y}_j - z_j),\label{eq:timeDiscSys_diffsigma2}
    \end{align}
\end{subequations}
with initial guess~$\widehat{y}_0$ for the observer and unknown initial condition~$y_0$ for the physical system.

\begin{remark}\label{re:appr}
    Alternatively to the exact representation via the matrix exponential, we can use approximations of $\mathcal{A}_\sigma$ and $\mathcal{B}_\sigma$. Below we list two typical choices.
\begin{itemize}
    \item   For the choices
            \begin{align*}
                \clA_\sigma \coloneqq (I +\Delta t A_\sigma),\quad \text{and}\quad \clB \coloneqq\Delta t B,
            \end{align*}
            the approximation~\eqref{eq:y_discr} corresponds to a forward Euler discretization of the physical system~\eqref{eq:sys-cltwin}.
    \item   For the choices
            \begin{align*}
                \clA_\sigma \coloneqq \left(I - \frac{\Delta t}{2} A_\sigma\right)^{-1} \left(I + \frac{\Delta t}{2} A_\sigma\right),\quad \text{and},\quad \clB_\sigma\coloneqq \Delta t \left(I - \frac{\Delta t}{2} A_\sigma\right)^{-1}B,
            \end{align*}
            the approximation~\eqref{eq:y_discr} corresponds to a semi-implicit Crank--Nicolson discretization of the physical system~\eqref{eq:y_cont}. 
\end{itemize}
\end{remark}


\section{Infinite horizon asymptotics}
\label{sec:asympt}

Here, we provide an analysis for system \eqref{eq:timeDiscSys_diffsigma} as $j\to \infty$. Throughout, $\sigma$ is the fixed but unknown coefficient in \eqref{sys-intro}. First, $\widehat{\sigma}$ remains constant, then in~\Cref{sec:subsec31} the special case $\widehat{\sigma} = \sigma$ is considered, and subsequently a sequence of estimates $(\widehat{\sigma}_i)_{i\in \mathbb{N}}$ is considered in~\Cref{sec:subsec32}. Throughout we will assume that the estimates are sufficiently close to $\sigma$. All operators $\clA_\varsigma,\clB_\varsigma,K_\varsigma,L_\varsigma$ in \eqref{eq:timeDiscSys_diffsigma} are assumed to depend  continuously on $\varsigma$ in a neighborhood of $\sigma$.

For our analysis, it is convenient to introduce the tracking error $e_j= y_j-\hat y_j$, for $j=0,1, \dots $.
Then, from \eqref{eq:timeDiscSys_diffsigma}, we deduce that
\begin{align*}
    y_{j+1} &= \clA_{\sigma} y_j + \clB_{\sigma} K_{\widehat{\sigma}} \widehat{y}_j =   \clA_{\sigma} y_j + \clB_{\sigma} K_{\widehat{\sigma}} {y}_j -  \clB_{\sigma} K_{\widehat{\sigma}} e_j\\
    \widehat{y}_{j+1} &= (\clA_{\widehat{\sigma}} + \clB_{\widehat{\sigma}}K_{\widehat{\sigma}})\widehat{y}_j + L_{\widehat{\sigma}}C(\widehat{y}_j - y_j) - L_{\widehat{\sigma}} \eta_j\\
    e_{j+1} &= (\clA_\sigma - \clA_{\widehat{\sigma}}) y_j + (\clA_{\widehat{\sigma}} + L_{\widehat{\sigma}} C )e_j + (\clB_\sigma K_{\widehat{\sigma}} - \clB_{\widehat{\sigma}}K_{\widehat{\sigma}})y_j -  (\clB_\sigma K_{\widehat{\sigma}} - \clB_{\widehat{\sigma}}K_{\widehat{\sigma}}) e_j + L_{\widehat{\sigma}}\eta_j,
\end{align*}
where the conditions on the noise sequence $\{\eta_j\}$ in $\bbR^\ell$ are given below.
Associated to this system, we define
\begin{align}\label{eq:F}
F =F(\sigma,\widehat \sigma)= \begin{bmatrix}
    \clA_\sigma + \clB_\sigma K_{\widehat{\sigma}}  & - \clB_\sigma K_{\widehat{\sigma}} \\
    \clA_\sigma - \clA_{\widehat{\sigma}} + (\clB_\sigma - \clB_{\widehat{\sigma}})K_{\widehat{\sigma}} & \clA_{\widehat{\sigma}} + L_{\widehat{\sigma}}C - (\clB_\sigma - \clB_{\widehat{\sigma}})K_{\widehat{\sigma}}
\end{bmatrix}.
\end{align}
Then, the iteration in the variables $(y_j,e_j)$ can be expressed as

\begin{align}\label{eq:iter}
    X_{j+1} = F(\sigma,\widehat \sigma) \, X_j - \begin{bmatrix} 0 \\ L_{\widehat{\sigma}}\end{bmatrix} \eta_j, \text{ with } X_0= \begin{bmatrix}y_0\\y_0 - \widehat{y}_0\end{bmatrix},
\end{align}
where
\begin{align*}
 X_j = \begin{bmatrix}y_j\\e_j\end{bmatrix} = \begin{bmatrix}y_j\\y_j - \widehat{y}_j\end{bmatrix}.
\end{align*}

Hereafter, we denote by~$\varrho(M)$ the spectral radius of a matrix~$M$.

\begin{theorem}\label{thm:cpl_contraction}
    Let $\sigma$ be fixed. Assume existence of an open neighborhood $\mathcal{I}$ of $\sigma$ and continuous, matrix valued functions $\varsigma \in \mathcal{I} \mapsto K_{\varsigma}$ and $\varsigma \in \mathcal{I} \mapsto L_{\varsigma}$, and that $\varrho( \clA_{{\sigma}}+\clB_\sigma K_{{\sigma}})<1$ and $\varrho(\clA_{{\sigma}}+L_{{\sigma}} C)<1$.
    Then, there exists $\mathcal{I}^\prime \subset \mathcal{I}$ such that for all $\varsigma \in \mathcal{I}^\prime$ it holds that
    $\varrho(F(\sigma,\varsigma)) < 1$.
\end{theorem}
\begin{remark} \label{rem:ricgains}
    Before presenting the proof of~\Cref{thm:cpl_contraction}, we point out that, under mild assumptions, the theorem can be applied with the gains obtained via standard discrete-time Riccati equations. 
    To that end, let $\sigma$ be fixed and assume that the pairs $(\mathcal{A}_\sigma,\mathcal{B}_\sigma)$ and $(\mathcal{A}_\sigma, C)$ are stabilizable and detectable, respectively.
    The continuity of $\mathcal{A}_\sigma$ and $\mathcal{B}_\sigma$ in $\sigma$ implies existence of $\mathcal{I}$ containing $\sigma$, such that, for all $\varsigma \in \mathcal{I}$ the pairs $(\mathcal{A}_{\varsigma},\mathcal{B}_{\varsigma})$ and $(\mathcal{A}_{\varsigma}, C)$ are stabilizable and detectable. Further, assume a sufficiently small time step, such that $\mathcal{A}_{\varsigma}$ is invertible. Then, for any symmetric positive definite matrices $R \in \mathbb{R}^{m\times m}$, $Q \in \mathbb{R}^{n\times n}$, the discrete algebraic Riccati equations 
    \begin{equation*}
    \begin{aligned}
        X_{\varsigma}
        &=
        \mathcal{A}_{\varsigma}^\top X_{\varsigma} \mathcal{A}_{\varsigma} 
        + 
        C^\top C
        -(\mathcal{B}_{\varsigma}^\top X_{\varsigma} \mathcal{A}_{\varsigma})^\top (R+\mathcal{B}_{\varsigma}^\top X_{\varsigma} \mathcal{B}_{\varsigma})^{-1}
        (\mathcal{B}_{\varsigma}^\top X_{\varsigma} \mathcal{A}_{\varsigma})\\
        Y_{\varsigma}
        &=
        \mathcal{A}_{\varsigma} Y_{\varsigma} \mathcal{A}_{\varsigma}^\top 
        +
        Q
        - (C Y_{\varsigma} \mathcal{A}_{\varsigma}^\top)^\top 
        (\Gamma + C Y_{\varsigma} C^\top)^{-1} C Y_{\varsigma} \mathcal{A}_{\varsigma}^\top
    \end{aligned}
    \end{equation*}
    admit unique symmetric, positive semidefinite, stabilizing solutions $X_{\varsigma}$ and $Y_{\varsigma}$, cf.~\cite[Cor.~13.5.3]{LanRod95}. Additionally, $X_{\varsigma}$ and $Y_{\varsigma}$ depend continuously on $\varsigma$, see \cite[Thm.~14.2.1]{LanRod95}. Hence, the associated gains 
    \begin{equation*}
    \begin{aligned}
        K_{\varsigma} 
        &=
        -(R + \mathcal{B}_{\varsigma}^\top  X_{\varsigma} \mathcal{B}_{\varsigma})^{-1} \mathcal{B}_{\varsigma}^\top X_{\varsigma} \mathcal{A}_{\varsigma},\\
        L_{\varsigma}
        &= 
        -\mathcal{A}_{\varsigma} Y_{\varsigma} C^\top
        (\Gamma + C Y_{\varsigma} C^\top )^{-1}
    \end{aligned}
    \end{equation*}
    satisfy the assumptions of \Cref{thm:cpl_contraction}.
\end{remark}
\begin{proof}[Proof of~\Cref{thm:cpl_contraction}]
    Let $\sigma$ be fixed and $\mathcal{I}$ as in the assumption. For now, set $\mathcal{I}^\prime = \mathcal{I}$. Define the continuous mappings
    \[
    \varsigma \mapsto
    \clA_1(\varsigma) := \clA_\sigma+\clB_{\sigma} K_{\varsigma},
    \qquad
    \varsigma \mapsto \clA_2(\varsigma) := \clA_{\varsigma}+L_{\varsigma}C + (\clB_\sigma - \clB_{\varsigma})K_{\varsigma},
    \]
    \[
    \varsigma \mapsto \clB_1(\varsigma) := -\clB_{\sigma}  K_{\varsigma},
    \qquad
    \varsigma \mapsto \clB_2(\varsigma) := \clA_{\sigma}-\clA_{\varsigma}- (\clB_\sigma - \clB_{\varsigma})K_{\varsigma}.
    \]
    By assumption, there holds $\varrho(\mathcal{A}_\sigma + \mathcal{B}_\sigma K_{\sigma}) < 1$ and due to the continuity of $K_{\sigma}$ a possible decrease of $\mathcal{I}^\prime$ ensures that for $\varsigma \in \mathcal{I}^\prime$ it follows that
    $\varrho (\mathcal{A}_1(\varsigma)) = \varrho(\mathcal{A}_\sigma + \mathcal{B}_\sigma K_{\varsigma}) < 1$. 
    Analogous arguments yield that $\varrho(\mathcal{A}_2(\varsigma)) < 1$ for $\varsigma \in \mathcal{I}^\prime$.
    We now proof the assertion under the technical assumption that there exists a submultiplicative norm $\Vert \cdot \Vert_*$ such that for all $\varsigma \in \mathcal{I}^\prime$ there holds
    \begin{equation}\label{eq:contr}
    \big\|(\lambda I-\clA_2(\varsigma))^{-1} 
    \clB_2(\varsigma) (\lambda I-\clA_1(\varsigma))^{-1}
    \clB_1(\varsigma) \big\|_* < 1
    \quad 
    \text{for all } \lambda \in \Lambda := \{\lambda \in \mathbb{C} : |\lambda|\ge 1\}.
    \end{equation}
    From now on, for ease of readability, we suppress the dependence on $\varsigma$ in the notation. 
    Observe that
    \[
    F(\sigma,\varsigma) = F = \begin{bmatrix} \clA_1 & \clB_1 \\[1mm] \clB_2 & \clA_2 \end{bmatrix}.
    \]
    Since $\varrho(\clA_i)<1$, the spectrum of $\clA_i$ is contained in the open unit disk. Hence, $\lambda I- \clA_i$ is invertible for all $\lambda \in \Lambda$, for $i=1,2$. For $\lambda \in \Lambda$ we write
    \begin{align*}
    \lambda I-F &= \begin{bmatrix} \lambda I-\clA_1 & -\clB_1 \\ -\clB_2 & \lambda I-\clA_2 \end{bmatrix}\\
    &=
    \begin{bmatrix} \lambda I-\clA_1 & 0 \\ -\clB_2 & \lambda I-\clA_2 \end{bmatrix}
    \begin{bmatrix} I & -(\lambda I-\clA_1)^{-1}\clB_1 \\[1mm] 0 & I - (\lambda I-\clA_2)^{-1} \clB_2 (\lambda I-\clA_1)^{-1} \clB_1 \end{bmatrix}.
    \end{align*}
    By hypothesis, there holds
    \[
    \big\|(\lambda I-\clA_2)^{-1} \clB_2 (\lambda I-\clA_1)^{-1} \clB_1\big\|_* < 1,
    \]
    so the lower-right block of the second factor is invertible via the Neumann series
    \[
    \left(I - (\lambda I-\clA_2)^{-1}\clB_2(\lambda I-\clA_1)^{-1}\clB_1\right)^{-1} = \sum_{i = 0}^{\infty} \left((\lambda I-\clA_2)^{-1}\clB_2(\lambda I-\clA_1)^{-1}\clB_1\right)^i.
    \]
    Hence, the second factor is invertible. The first factor is invertible as it is block lower-triangular with invertible diagonal blocks. Therefore, $\lambda I-F$ is invertible for all $\lambda \in \Lambda$, so no eigenvalue of $F$ lies on or outside the unit circle. We conclude that $\varrho(F) <1$.
    
    It remains to show the existence of norms as in \eqref{eq:contr}.
    First, note that $\clA_1(\sigma)= \clA_\sigma+\clB_{\sigma} K_{{\sigma}}$, $\clA_2(\sigma)= \clA_\sigma+ L_{\sigma}C$, and $\clB_2(\sigma)=0$.
    Now, for $\lambda \in \Lambda$ and $\varsigma \in \mathcal{I}^\prime$, define 
    \[
    G(\lambda,\varsigma) := (\lambda I-\clA_2(\varsigma))^{-1} \clB_2(\varsigma) (\lambda I-\clA_1(\varsigma))^{-1} \clB_1(\varsigma).
     \]
    Since $\varrho(\clA_i(\sigma))<1$, there exists norms $\Vert \cdot \Vert_*$ and $\Vert \cdot \Vert_\dagger$ on $\mathbb{R}^n$, such that, for their respective induced matrix norms there holds
    $\Vert \mathcal{A}_1(\sigma) \Vert_\dagger <1$
    and
    $\Vert \mathcal{A}_2(\sigma) \Vert_* < 1$, \cite[Sect.~6.9, (6.9.2) Thm.]{StoBul93}.
    Also, possibly after further reducing $\mathcal{I}^\prime$, we have that $\Vert \clA_1(\varsigma) \Vert_\dagger, \Vert \mathcal{A}_2(\varsigma) \Vert_* <1$, for all $\varsigma \in \mathcal{I}^\prime$. 
    Since the matrix norms $\Vert \cdot \Vert_\dagger$ and $\Vert \cdot \Vert_*$ are equivalent,
    there exists $K>0$, such that $\dnorm{\cdot}{*} \le K\dnorm{\cdot}{\dagger}$.
    Consequently, for $\lambda \in \Lambda$, we estimate
    \begin{equation*}
    \begin{array}l
    \displaystyle
    \dnorm{G(\lambda,\varsigma)}{*} 
    \le 
    \frac{K \dnorm{\clB_2(\varsigma)}{*}}{|\lambda| - \dnorm{\clA_2(\varsigma)}{*}}
    \dnorm{  (\lambda I-\clA_1(\varsigma))^{-1}
    \clB_1(\varsigma)}{\dagger} \\ \\
    \displaystyle
    \le
    \frac{K \dnorm{\clB_2(\varsigma)}{*}  \dnorm{\clB_1(\varsigma)}{\dagger}}
    {(|\lambda|- \dnorm{\clA_2(\varsigma)}{*})(|\lambda| -\dnorm{\clA_1(\varsigma)}{\dagger})}
    \le
    \frac{K \dnorm{\clB_2(\varsigma)}{*}  \dnorm{\clB_1(\varsigma)}{\dagger}}
    {(1- \dnorm{\clA_2(\varsigma)}{*})(1 -\dnorm{\clA_1(\varsigma)}{\dagger})},
    \end{array}
    \end{equation*}
    and hence, due to the form of $\clB_2$,  after further reducing $\mathcal{I}^\prime$, we have that
    $\dnorm{G(\lambda,\varsigma)}{*} < 1$  and thus, \eqref{eq:contr} holds for all $\lambda\in \Lambda$. As an induced matrix norm $\Vert \cdot \Vert_*$ is submultiplicative, the proof is finished.
\end{proof}

\begin{remark}
If  the operators $\clA_\sigma$ and $\clB_\sigma$ introduced below \eqref{eq:y_discr} are approximated by 
 $\tilde \clA_\sigma$ and $\tilde \clB_\sigma$, as discussed in \Cref{re:appr}, then \Cref{thm:cpl_contraction} is again applicable with $\clA_\sigma$ and $\clB_\sigma$ replaced by $\tilde \clA_\sigma$ and $\tilde \clB_\sigma$, and analogously for  $\clA_{\widehat \sigma}$ and $\clB_{\widehat \sigma}$. Indeed, for this purpose, one can again validate the steps in the proof of \Cref{thm:cpl_contraction} provided that 
  $\|\clA_\sigma- \tilde \clA_{\sigma}\|$, $\|\clB_\sigma-\clB_{\widehat \sigma}\|$, $\|\clA_{\widehat \sigma}- \tilde \clA_{\widehat \sigma}\|$, and  $\|\clB_{\widehat \sigma}- \tilde \clB_{\widehat \sigma}\|$ are sufficiently small. 
\end{remark}
We turn to analyzing the convergence in distribution of system \eqref{eq:timeDiscSys_diffsigma}. First, the definition adapted to our setting is recalled.  
\begin{definition}
    Let $X_k$ be a sequence of normally distributed random vectors with $X_k \sim \mathcal{N}(m_k,\Sigma_k)$. We say that $X_k$ converges in distribution to the Gaussian variable $X \sim \mathcal{N}(m,\Sigma)$ if
    $m_k$ and $\Sigma_k$ converge to $m$ and $\Sigma$, respectively.
\end{definition}
\begin{remark}
    For a more general definition of convergence in distribution see \cite[Sec.~25]{Bil95}. By \cite[Thm.~30.2, Ex.~30.1]{Bil95}, for Gaussian random vectors, it is equivalent to our definition.
\end{remark}

We henceforth assume that the initial condition $y_0$ of the state variable is distributed according to a Gaussian $\mathcal{N}(m_0,\mathcal{C}_0)$.  For the measurement error we take  $\eta_j \sim \mathcal{N}(0,\Gamma_j)$ mutually independent and $\Gamma_j \to \Gamma_\infty$. We shall also need the covariance matrix
\begin{equation}\label{eq:Pi_0}
\mathbf{\Pi}_0 = \left[
                      \begin{array}{cc}
                        \mathcal{C}_0 & \mathcal{C}_0 \\
                        \mathcal{C}_0 & \mathcal{C}_0 \\
                      \end{array}
                    \right].
\end{equation}

\begin{theorem}\label{thm:conv_in_dist_obsErr}
    Given $\sigma \in \mathbb{R}^p$ and $\widehat \sigma \in \mathbb{R}^p$, consider system \eqref{eq:timeDiscSys_diffsigma} with the assumptions in \Cref{thm:cpl_contraction} holding for $\sigma, \widehat \sigma$ and initialize the virtual twin with $\widehat{y}_0 = m_0$.
    
    Then, the iterates $X_j$ of \eqref{eq:timeDiscSys_diffsigma} are  normally distributed at each iteration level, and $X_j \sim \mathcal{N}(F^j \bigl[\begin{smallmatrix} m_0 \\ 0 \end{smallmatrix}\bigr] ,\mathbf{\Pi}_j)$, where
    \begin{equation*}
        \mathbf{\Pi}_j = F(\sigma,\widehat \sigma)^j \, \mathbf{\Pi}_0 \, {F(\sigma,\widehat \sigma)^\top}^j
        +
        \sum_{k=0}^{j-1} F(\sigma,\widehat \sigma)^{j-1-k}
        \,  \bigl[\begin{smallmatrix} 0\\ L_{\widehat \sigma} \end{smallmatrix}\bigr] {\Gamma_k}  \bigl[\begin{smallmatrix} 0\\ L_{\widehat \sigma} \end{smallmatrix}\bigr]^\top \, {F(\sigma,\widehat \sigma)^\top\,^{j-1-k}},
    \end{equation*}%
    for $j\ge1$.
    Further, for $j \to \infty $, the pair $(y_j,e_j)$ converges  to
    \begin{equation*}
        (y_\infty,e_\infty) \sim \mathcal{N}(0,\mathbf{\mathbf{\Pi_\infty}})
    \end{equation*}
    in distribution, where $\mathbf{\Pi_\infty}$ is  the unique solution to the Lyapunov equation
    \begin{equation}\label{eq:kk1}
        \mathbf{\Pi_\infty} = F(\sigma,\widehat \sigma) \mathbf{\Pi_\infty} F(\sigma,\widehat \sigma)^\top +  \bigl[\begin{smallmatrix} 0\\ L_{\widehat \sigma} \end{smallmatrix}\bigr] \Gamma_\infty  \bigl[\begin{smallmatrix} 0\\ L_{\widehat \sigma} \end{smallmatrix}\bigr]^\top.
    \end{equation}
\end{theorem}
\begin{proof}
    From \eqref{eq:iter}, for $j \geq 0$, we obtain the recursion formula
    \begin{equation*}
        X_j
        =
        F(\sigma,\widehat \sigma)^j \bigl[\begin{smallmatrix} y_0 \\ y_0-m_0 \end{smallmatrix}\bigr] - \sum_{k=0}^{j-1} F(\sigma,\widehat \sigma)^{j-1-k}  \bigl[\begin{smallmatrix} 0 \\ L_{\widehat \sigma} \end{smallmatrix}\bigr] \eta_k.
    \end{equation*}
    The independence of $y_0$ and the $\eta_j$ yields the characterization of $X_j$. Since  $\Gamma_j \to \Gamma_\infty$ and $F(\sigma, \widehat \sigma)$ is a contraction, by \Cref{thm:cpl_contraction}, a minor extension of the arguments in \cite[Lemma 4.2, Corollary 2]{soed}
    allows us to conclude that $\mathbf{\Pi}_j$ converges to the solution of the Lyapunov equation \eqref{eq:kk1}. Uniqueness of $\mathbf{\Pi}_\infty$ follows from the contraction property of $F$.  This ends the proof. 
\end{proof}
We consider next the case of decaying measurement noise.
\begin{corollary}
    In addition to the assumptions of \Cref{thm:conv_in_dist_obsErr}, let $\Gamma_j \to 0$, $j \to \infty$. Then, for any $\epsilon > 0$, there holds
    \begin{equation*}
        P(\Vert y_j \Vert + \Vert e_j \Vert > \epsilon) \to 0,~~~ j \to \infty.
    \end{equation*}
\end{corollary}
\begin{proof}
    Again we denote $\begin{bmatrix} y_j \\ e_j \end{bmatrix} = X_j$. From the previous theorem, we have that $X_j \sim \mathcal{N}(m_j,\mathbf{\Pi}_j)$, with $m_j \to 0$ for $j \to \infty$. Concerning $\mathbf{\Pi}_j$, we insert $\Gamma_\infty = 0$ into \eqref{eq:kk1}, and utilizing the uniqueness of the associated solution, we obtain that $\mathbf{\Pi}_j$ converges to $\mathbf{\Pi}_\infty = 0$.
    Utilizing Markov's inequality for every $\epsilon > 0$, we obtain that
    \begin{equation*}
        P \left( \Vert X_j \Vert > \epsilon \right) 
        \leq 
        \epsilon^{-2} ~ \mathbb{E}\left[ \Vert X_j \Vert^2 \right]
        =
        \epsilon^{-2} ~ \left( \Vert m_j \Vert^2 + \mathrm{tr}(\mathbf{\Pi}_j) \right). 
    \end{equation*}
    As $\mathbf{\Pi}_j$ converges to zero, so does its trace $\mathrm{tr}(\mathbf{\Pi}_j)$, and the assertion is shown. 
\end{proof}

\subsection{The case~$\widehat\sigma=\sigma$}\label{sec:subsec31} The following corollary addresses the case that the estimated parameter $\widehat \sigma$ coincides with the true parameter of the physical system. For convenience, we repeat the expression for $F$ which in this case reads
$$
F =F(\sigma,\sigma)= \begin{bmatrix}
    \clA_\sigma + \clB_\sigma K_{{\sigma}}  & - \clB_\sigma K_{{\sigma}} \\
    0& \clA_{{\sigma}} + L_{{\sigma}}C 
\end{bmatrix}.
$$

Further, for $j = 1,\dots$, there holds
\begin{equation}\label{eq:F^j}
    F^j = \begin{bmatrix}
        (\mathcal{A}_{\sigma} + \mathcal{B}_{\sigma} K_\sigma )^j 
        & S_j \\
        0 & (\mathcal{A}_{\sigma} + L_\sigma C)^j
    \end{bmatrix},
\end{equation}
where 
$S_j = - \sum_{k=0}^{j-1} (\mathcal{A}_{\sigma} + \mathcal{B}_{\sigma} K_\sigma)^k
\mathcal{B}_{\sigma} K_\sigma 
(\mathcal{A}_{\sigma} + L_\sigma C)^{j-1-k}$.

\begin{corollary}
    Consider the setting of the previous theorem, and, in addition, assume that $\widehat \sigma = \sigma$. Then, we have $y_j \sim \mathcal{N}(m_j,\Pi^y_j)$ and $e_j \sim \mathcal{N}(0,\Pi^e_j)$, where
    \begin{equation*}
    \begin{aligned}
        m_j &= (\clA_\sigma + \clB_{\sigma} K_\sigma)^j m_0 \\
        \Pi^e_j &= (\clA_\sigma+ L_\sigma C)^j
        {\mathcal{C}_0} {(\clA_\sigma+ L_\sigma C)^j}^\top
        \\
        &+\sum_{k=0}^{j-1} (\clA_\sigma+ L_\sigma C)^{j-1-k} L_\sigma \Gamma_k L_\sigma^\top
       {(\clA_\sigma+ L_\sigma C)^{j-1-k}}^\top\\
        \Pi^y_j &= ((\clA_\sigma+ \clB_{\sigma} K_\sigma)^j+S_j)
        {\mathcal{C}_0} {((\clA_\sigma+ \clB_{\sigma} K_\sigma)^j+S_j)}^\top
        \\
        &+\sum_{k=0}^{j-1} S_{j-1-k} {L_\sigma} \Gamma_k {L_\sigma}^\top S_{j-1-k}^\top,
    \end{aligned}
    \end{equation*}
    for $j\ge1$, with $S_i$ as defined above. Further, we have convergence in distribution of $e_j$ and $y_j$  to
     $e_\infty \sim \mathcal{N}(0, \Pi^e_\infty)$ and
    $y_\infty \sim \mathcal{N}(0, \Pi^y_\infty)$, where $\Pi^e_\infty$ and $\Pi_\infty^y$ satisfy
    \begin{equation} \label{eq:rice_a}
        \Pi^e_\infty
        = (\clA_{\sigma} + L_\sigma C ) \Pi^e_\infty (\clA_{\sigma} + L_\sigma C)^\top
        +L_\sigma \Gamma_\infty L_\sigma^\top  ,
    \end{equation}
    and 
    \begin{equation} \label{eq:ricy_a}
    \begin{array}{ll}
        \Pi_\infty^y &= (\clA_{\sigma} + \clB_{\sigma} K_\sigma) \Pi^y_\infty (\clA_{\sigma} + \clB_{\sigma} K_\sigma)^\top
        + \clB_{\sigma} K_\sigma \Pi^e_\infty K_\sigma^\top \clB_{\sigma}^\top \\[1.5ex]
         &- (\clA_{\sigma} + \clB_{\sigma} K_\sigma)\Pi_\infty^{ye}(\clB_{\sigma} K_\sigma)^\top  - (\clB_{\sigma} K_\sigma) \Pi^{ey}_\infty (\clA_{\sigma} + \clB_{\sigma} K_\sigma)^\top,
    \end{array}    
    \end{equation}
    with 
    \begin{equation*}
        \Pi_\infty^{ye} ={\Pi^{ey}_\infty}^\top 
        = (\clA_{\sigma} + \clB_{\sigma} K_\sigma)\Pi^{ye}_\infty (\clA_{\sigma} + L_\sigma C)^\top - \clB_{\sigma} K_\sigma \Pi^e_\infty (\clA_{\sigma} + L_\sigma C)^\top.
    \end{equation*}
\end{corollary}
\begin{proof}
    For the sake of readability, within this proof, we drop the dependence on $\sigma$ in the notation.
    Applying \Cref{thm:conv_in_dist_obsErr}, for $j \geq 1$, we obtain that
    \begin{equation*}
        X_j
        =
        \begin{bmatrix}
            y_j \\ e_j   
        \end{bmatrix}
        \sim
        \mathcal{N}(F^j \bigl[ \begin{smallmatrix} m_0 \\ 0 \end{smallmatrix} \bigr], \mathbf{\Pi}_j).
    \end{equation*}
    Utilizing \eqref{eq:F^j}, the announced expected values of $y_j$ and $e_j$ follow immediately. We proceed by characterizing $\mathbf{\Pi}_j$ via
    \begin{equation*}
        \mathbf{\Pi}_j 
        =
        F^j \, \mathbf{\Pi}_0 \, {F^\top}^j
        +
        \sum_{k=0}^{j-1} F^{j-1-k}
        \,  \bigl[\begin{smallmatrix} 0\\ L\end{smallmatrix}\bigr] {\Gamma_k}
        \bigl[\begin{smallmatrix} 0\\ L\end{smallmatrix}\bigr]^\top \, {F^\top\,^{j-1-k}}.
    \end{equation*}
    After plugging in \eqref{eq:Pi_0} and \eqref{eq:F^j}, basic calculus yields that
    \begin{equation*}
        \mathbf{\Pi}_j
        =
        \begin{bmatrix}
            \Pi_j^{(1,1)} & \Pi_j^{(1,2)} \\
            \Pi_j^{(2,1)} & \Pi_j^{(2,2)}
        \end{bmatrix},
    \end{equation*}
    with 
    \begin{equation*}
    \begin{aligned}
        {\Pi}_j^{(1,1)}
        &=
        \left( (\mathcal{A} + \mathcal{B} K)^j + S_k \right)
        \mathcal{C}_0 \left( (\mathcal{A} + \mathcal{B} K)^j + S_k \right)^\top
        - \sum_{k=1}^{j-1} 
        S_{j-1-k} L \Gamma_k L^\top S_{j-1-k}^\top,\\
        {\Pi}_j^{(2,2)}
        &=
        (\mathcal{A} + L C) \mathcal{C}_0 (\mathcal{A} + L C)^\top
        - \sum_{k=1}^{j-1} 
        (\mathcal{A} + L C)^{j-1-k} L \Gamma_k L^\top {(\mathcal{A} + L C)^{j-1-k}}^\top,\\
        {\Pi}_j^{(1,2)}
        &=
        {{\Pi}_j^{(2,1)}}^\top\\
        &=
        (\mathcal{A} + L C)^j \mathcal{C}_0
        \left( (\mathcal{A} + \mathcal{B} K)^j + S_k \right)^\top
        - \sum_{k=1}^{j-1} 
        (\mathcal{A} + L C)^{j-1-k} L \Gamma_k L^\top S_{j-1-k}^\top,
    \end{aligned}
    \end{equation*}
    and the covariances of $y_j$ and $e_j$ are given by the diagonal blocks $\Pi_j^y = {\Pi}_j^{(1,1)}$ and $\Pi_j^e ={\Pi}_j^{(2,2)}$, respectively. 

    We turn to the verification of \eqref{eq:rice_a} and  \eqref{eq:ricy_a} and express \eqref{eq:kk1} as 
    \begin{equation*}  
    \begin{array} l
    \mathbf{\Pi}_\infty 
    =
    \left[
    \begin{array}{cc}
    \Pi^y_\infty & \Pi^{ye}_\infty  \\
    \Pi^{ey}_\infty & \Pi^e_\infty 
    \end{array}
    \right]
    = \\ \\
    \left[ \!\!\!
    \begin{array}{cc}
    \clA 
    +
    \clB K  \!\!\! & \!\!\! - \clB K \\
    0& \clA 
    +
    LC 
    \end{array}
    \!\!\! \right]    
    \!\!\!\left[\!\!\!
    \begin{array}{cc}
    \Pi^y_\infty  \!\!& \!\! \Pi^{ye}_\infty  \\
    \Pi^{ey}_\infty  \!\!& \!\! \Pi^e_\infty 
    \end{array}
    \!\!\! \right]   
    \!\!\! \left[ \!\!\!
    \begin{array}{cc}
    (\clA + \clB K)^\top   \!\!\!& \!\!\!0 \\
    (- \clB K)^\top \!\!\!&\!\!\! (\clA + LC)^\top 
    \end{array}
    \!\!\! \right]   +
    \left[ \!\!\!
    \begin{array}{cc}
    0 \!\!\!&\!\!\! 0\\
    0\!\!\!&\!  L \Gamma_\infty L^\top
    \end{array}
    \!\!\! \right].   
    \end{array}     
    \end{equation*}                 
    Solving this equation first for $\Pi^e_\infty$, we obtain \eqref{eq:rice_a}. Then, we solve for $\Pi^{ye}_\infty$, and finally for $\Pi^y_\infty$ to obtain \eqref{eq:ricy_a}.
\end{proof}

\subsection{The case~$\widehat\sigma$ is piecewise constant}\label{sec:subsec32}

Eventually, we want to update the estimate of the uncertain parameter online. We point out that switching between stable systems may have destabilizing effects \cite[Prop.~1]{AkarPaulSafoMitra06}, \cite[Prob.~A]{LiberzonMorse99}, \cite[Sect.~4.3.3]{GuthKunRod25MCRF}. This does not happen if~$\widehat \sigma$ remains in a sufficiently small neighborhood~$\clI$ of~$\sigma$, because $\varrho (\clA_\sigma+\clB_\sigma K_\sigma)<1$ implies~$\dnorm{\clA_\sigma+\clB_\sigma K_\sigma}{*}<1$ for some induced operator norm and, by continuity, $\dnorm{\clA_{\widehat\sigma}+\clB_{\widehat\sigma} K_{\widehat\sigma}}{*}<1$ (by taking a smaller neighborhood $\widetilde \clI\subset\clI$, if necessary).

Fix~$\Delta t>0$, we shall consider the case where~$\widehat\sigma$ is updated at instants of time multiples of~$\Delta t$:
\[\widehat\sigma(t)=\widehat\sigma_i,\qquad i\Delta t\le t<(i+1)\Delta t,\quad i\in\bbN.\]

Now, let
\[
 F_j =F(\sigma,\widehat\sigma_j)= \begin{bmatrix}
    \clA_\sigma + \clB_{\widehat\sigma_j} K_{\widehat\sigma_j}  & - \clB_{\widehat\sigma_j} K_{\widehat\sigma_j} \\
    0& \clA_{\widehat\sigma_j} + L_{\widehat\sigma_j}C 
\end{bmatrix},\qquad j\in\bbN.
\]

We have the following generalization of~\Cref{thm:conv_in_dist_obsErr}.
\begin{theorem}\label{thm:bdd_in_dist_obsErr-piececonst}
    Given $\sigma \in \mathbb{R}^p$ and $\widehat \sigma_i \in \mathbb{R}^p$, $i\in\bbN$, consider system \eqref{eq:timeDiscSys_diffsigma} with the assumptions in \Cref{thm:cpl_contraction} holding in a neighborhood of~$\sigma$. Initialize the virtual twin with $\widehat{y}_0 = m_0$.
     Then, the iterates $X_j$ of \eqref{eq:timeDiscSys_diffsigma} are  normally distributed at each iteration level, and $X_j \sim \mathcal{N}(\Phi_{j|0} \bigl[\begin{smallmatrix} m_0 \\ 0 \end{smallmatrix}\bigr] ,\mathbf{\Pi}_j)$, where
    \begin{align*}
   & \Phi_{0|i}=\Id; \quad \Phi_{m+1|i}\coloneqq F_{i+m}\Phi_{m|i},\quad &&i\in\bbN,\quad m\in\bbN;\\
           & \mathbf{\Pi}_j = \Phi_{j|0} \, \mathbf{\Pi}_0 \, \Phi_{j|0}^\top
        +
        \sum_{k=0}^{j-1} \Phi_{j-1-k| k}
        \,  \bigl[\begin{smallmatrix} 0\\ L_{\widehat \sigma_k} \end{smallmatrix}\bigr] {\Gamma_k}  \bigl[\begin{smallmatrix} 0\\ L_{\widehat \sigma_k} \end{smallmatrix}\bigr]^\top \, \Phi_{j-1-k|k}^\top,\qquad &&j\in\bbN_+.
      \end{align*}
    Further, there exists a matrix norm~$\dnorm{\cdot}{}$ and~$\lambda<1$ such that, for every~$\widehat\sigma_j$ in a small neighborhood of~$\sigma$,
    \begin{equation}\label{eq:kk1-piececonst}
    \begin{split}
    & \dnorm{\Phi_{j|0}\bigl[\begin{smallmatrix} m_0\\ 0\end{smallmatrix}\bigr]}{}\le \lambda^j \dnorm{\bigl[\begin{smallmatrix} m_0\\ 0\end{smallmatrix}\bigr]}{};\\
       & \dnorm{\mathbf{\Pi}_j}{}\le \lambda^{2j} \dnorm{\mathbf{\Pi}_0}{}+\frac1{1-\lambda^2}  \sup_{k\in\bbN} \dnorm{\bigl[\begin{smallmatrix} 0\\ L_{\widehat \sigma_k} \end{smallmatrix}\bigr] {\Gamma_k}  \bigl[\begin{smallmatrix} 0\\ L_{\widehat \sigma_k} \end{smallmatrix}\bigr]^\top}{} .
       \end{split}
    \end{equation}
\end{theorem}

\begin{proof}
    From the analogue to~\eqref{eq:iter}, 
\begin{align}\label{eq:iter-pcest}
   X_0= \begin{bmatrix}y_0\\y_0 - \widehat{y}_0\end{bmatrix},\qquad X_{j+1} = F(\sigma,\widehat \sigma_j) \, X_j - \begin{bmatrix} 0 \\ L_{\widehat{\sigma}_j}\end{bmatrix} \eta_j, \quad j\in\bbN,
\end{align}
  we find, for~$j\ge1$,
    \begin{align*}
        X_{j+1}
        &=
        F(\sigma,\widehat \sigma_j) \left(F(\sigma,\widehat \sigma_{j-1}) \, X_{j-1} - \begin{bmatrix} 0 \\ L_{\widehat{\sigma}_{j-1}}\end{bmatrix} \eta_{j-1}\right)-\begin{bmatrix} 0 \\ L_{\widehat{\sigma}_{j}}\end{bmatrix} \eta_{j}\\
        &=
        F(\sigma,\widehat \sigma_j) F(\sigma,\widehat \sigma_{j-1}) \, X_{j-1} - F(\sigma,\widehat \sigma_j)\begin{bmatrix} 0 \\ L_{\widehat{\sigma}_{j-1}}\end{bmatrix} \eta_{j-1}-\begin{bmatrix} 0 \\ L_{\widehat{\sigma}_{j}}\end{bmatrix} \eta_{j}\\
        &=
        \Phi_{2|j-1} \, X_{j-1} - \Phi_{1|j}\begin{bmatrix} 0 \\ L_{\widehat{\sigma}_{j-1}}\end{bmatrix} \eta_{j-1}-\begin{bmatrix} 0 \\ L_{\widehat{\sigma}_{j}}\end{bmatrix} \eta_{j}.
    \end{align*}
    Repeating the argument, for~$j\ge2$,
    \begin{align*}
        X_{j+1}
        &=
        \Phi_{3|j-2} \, X_{j-2} - \Phi_{2|j-1}\begin{bmatrix} 0 \\ L_{\widehat{\sigma}_{j-2}}\end{bmatrix} \eta_{j-2} - \Phi_{1|j}\begin{bmatrix} 0 \\ L_{\widehat{\sigma}_{j-1}}\end{bmatrix} \eta_{j-1}-\begin{bmatrix} 0 \\ L_{\widehat{\sigma}_{j}}\end{bmatrix} \eta_{j}\\
        &=\Phi_{j+1|0} \, X_{0} - \sum_{k=0}^{j}\Phi_{k|j-k+1}\begin{bmatrix} 0 \\ L_{\widehat{\sigma}_{j-k}}\end{bmatrix} \eta_{j-k}.
    \end{align*}
Following again the arguments in \cite[Lem.~4.2, Rem.~3, Eq.~(4.18)]{soed}, now in the context of time-varying systems, we find~$X_{j}=\clN(M_{j},\mathbf{\Pi}_{j}),$ where
\begin{align*}
M_{j}&=\Phi_{j|0} \bigl[\begin{smallmatrix} m_0\\ 0\end{smallmatrix}\bigr],
\qquad
R_{j}=\bigl[\begin{smallmatrix} 0\\ L_{\widehat \sigma_j} \end{smallmatrix}\bigr] {\Gamma_j}  \bigl[\begin{smallmatrix} 0\\ L_{\widehat \sigma_j} \end{smallmatrix}\bigr]^\top,
\intertext{and}
 \mathbf{\Pi}_{j}&= \Phi_{1|j-1}\, \mathbf{\Pi}_{j-1} \, \Phi_{1|j-1}^\top + R_{j-1}= \Phi_{2|j-2} \, \mathbf{\Pi}_{j-2} \, \Phi_{2|j-2}^\top + \Phi_{1|j-1} \, R_{j-2} \,\Phi_{1|j-1}^\top +R_{j-1}\\
 &= \Phi_{3|j-3} \, \mathbf{\Pi}_{j-3} \, \Phi_{3|j-3}^\top + \Phi_{2|j-2} \, R_{j-3} \,\Phi_{2|j-3}^\top + \Phi_{1|j-1} \, R_{j-2} \,\Phi_{1|j-1}^\top +R{j-1},\\
 &= \Phi_{j|0} \, \mathbf{\Pi}_{0} \, \Phi_{j|0}^\top +\sum_{k=0}^{j-1} \Phi_{j-k-1|k+1} \, R_{k} \,\Phi_{j-k-1|k+1}^\top.
\end{align*}

By \Cref{thm:cpl_contraction} we have that~$\varrho(F(\sigma,\sigma))<1$, which implies that there exists an induced matrix~$\dnorm{\cdot}{}$  such that~$\dnorm{F(\sigma,\sigma))}{}<1$. By continuity,~$\dnorm{F(\sigma,\varsigma))}{}\leq \lambda <1$ for every~$\varsigma$ in a small neighborhood of~$\sigma$, where $\lambda < 1$. Therefore, for~$(\widehat\sigma_j)_{j\in\bbN}$ contained in that neighborhood of~$\sigma$, for the mean we obtain
\[
 \dnorm{M_j}{}\le\dnorm{\Phi_{j|0}}{}\dnorm{\bigl[\begin{smallmatrix} m_0\\ 0\end{smallmatrix}\bigr]}{}\le\lambda^j\dnorm{\bigl[\begin{smallmatrix} m_0\\ 0\end{smallmatrix}\bigr]}{},
\]
and for the covariance,
\begin{align*}
 \dnorm{\mathbf{\Pi}_j}{}&\le
\dnorm{\Phi_{j|0}}{}^2\dnorm{\mathbf{\Pi}_{0}}{} +\sum_{k=0}^{j-1}\dnorm{ \Phi_{j-k-1|k+1}}{}^2 \, \dnorm{R_{k}}{}\le
\lambda^{2j}\dnorm{\mathbf{\Pi}_{0}}{} +\sum_{k=0}^{j-1}\lambda^{2(j-k-1)}\dnorm{R_{k}}{}\\
&\le
\lambda^{2j}\dnorm{\mathbf{\Pi}_{0}}{} +\frac{1}{1-\lambda^2}\sup_{k\in\bbN}\dnorm{R_{k}}{},
\end{align*}
which finishes the proof.
\end{proof}

In practice, the parameter will not be updated at every time step, but only at multiples $t = \kappa s$ of $s \in \mathbb{N}_+$, $\kappa \in \mathbb{N}$, cf., the bullet points in \Cref{fig:init_est} below.

\section{Parameter estimation}
\label{sec:param}
For \Cref{thm:conv_in_dist_obsErr} above, it is essential to have a good estimate~$\widehat{\sigma}$ of the unknown parameter~$\sigma \in \mathbb{R}^p$. We adopt a Bayesian perspective: starting with a prior distribution of the unknown parameter, we update this prior belief sequentially based on new data from the controlled physical-virtual-twin system. We denote the estimate after the $k$-th update by $\widehat{\sigma}_k$, resulting in a sequence of estimates $(\widehat{\sigma}_k)_{k
\ge0}$. Each $\widehat{\sigma}_k$ is computed as the posterior mean from a Bayesian inverse problem using time-discrete, noisy, partial observations of the physical twin. This results in two nested time grids, as illustrated in \Cref{fig:init_est}: a sequence of parameter estimation intervals indexed by~$k\ge1$, each of which contains a sequence of observations~$z_j$, $j=0,\ldots,s$. For each Bayesian update, a distribution of the unknown initial state $y_0^{(k)}$ conditioned on $\sigma$ in the corresponding interval is required. It turns out that, this distribution is Gaussian and can be obtained by Kalman filtering as outlined in \Cref{sec:Kalmanfilter} below. This section is structured as follows: we commence with details on a single Bayesian parameter update for our specific model in ~\Cref{sec:singleBayes} and \Cref{sec:forwardmap}. In \Cref{sec:para_est2}, we present an illustrative explanation for the parameter estimation for two consecutive estimation intervals before we explain how the method is applied sequentially for an arbitrary number of estimation intervals in \Cref{sec:para_est_general}.

\subsection{Bayesian inverse problem}\label{sec:singleBayes}
In this and the subsequent subsection, we focus on a single parameter estimation, hence, $k$ and $\widehat{\sigma}^{(k)}$ are fixed. To improve clarity, we omit the superscript notation indicating the $k$-th estimation interval. Let us assume a joint prior measure of the unknown initial condition $y_0$ and the unknown parameter $\sigma$
\begin{align*}
\mu_{\mathrm{prior}}(\mathrm d\sigma,\mathrm dy_0) \,=\,\mu_{\mathrm{prior}}(\mathrm dy_0\mid\sigma)\,\mu_{\mathrm{prior}}(\mathrm d\sigma),
\end{align*}
where $\mu_{\mathrm{prior}}(\mathrm d\sigma)$ is a probability measure on~$\mathbb R^p$ (with Lebesgue density~$\rho_{\rm prior}(\sigma)$),
and $\mu_{\mathrm{prior}}(\mathrm dy_0\mid\sigma)$ is a probability measure on~$\bbR^n$. The joint posterior measure 
on~$(\sigma,y_0)$ given data~$\bsz = \begin{bmatrix}z_0^\top&\ldots &z_s^\top\end{bmatrix}^\top$ is {determined} in the sense of Radon--Nikodym derivatives by
\begin{align*}
\frac{\mathrm d\mu_{\mathrm{post}}(\cdot \mid \bsz)}{\mathrm d\mu_{\mathrm{prior}}}(\sigma, y_0) \;\propto\; \rho_{\mathrm{like}}(\bsz \mid \sigma, y_0),
\end{align*}
provided that $\rho_{\rm like}$ is nonnegative and belongs to $L^1(\mu_{\rm prior})$. Equivalently, we can write
\begin{align*}
\mu_{\mathrm{post}}(\mathrm d\sigma,\mathrm dy_0\mid \bsz)
&=\frac{1}{Z(\bsz)}\,
\rho_{\mathrm{like}}(\bsz\mid\sigma,y_0)\,
\mu_{\mathrm{prior}}(\mathrm dy_0\mid\sigma)\,
\mu_{\mathrm{prior}}(\mathrm d\sigma),
\\
Z(\bsz)
&=\int_{\bbR^p} \int_{\bbR^n} \rho_{\mathrm{like}}(\bsz\mid\sigma,y_0)\,
\mu_{\mathrm{prior}}(\mathrm dy_0\mid\sigma)\,
\mu_{\mathrm{prior}}(\mathrm d\sigma).
\end{align*}
Here, $0<Z(\bsz)<\infty$ is the normalizing constant (sometimes called model evidence). Computing the joint posterior is challenging as 
the forward map~$(\sigma,y_0) \mapsto \bsz$ is nonlinear and only implicitly defined through the solution of the coupled dynamical system, which makes the posterior non-Gaussian. 

For the stabilization problem, the initial condition acts as a nuisance parameter: it is not of direct interest for the output-feedback design but enters the likelihood through the model dynamics and therefore must be accounted for in the inference. For this reason, we pursue inference on~$\sigma$ based on its marginal posterior obtained by integrating out the latent variable~$y_0$
\begin{align*}
\mu_{\mathrm{post}}(\mathrm d\sigma\mid \bsz)
\;=\;
\frac{1}{Z(\bsz)}
\Bigg(
\int_{\bbR^n} \rho_{\mathrm{like}}(\bsz\mid\sigma,y_0)\,
\mu_{\mathrm{prior}}(\mathrm dy_0\mid\sigma)
\Bigg)
\mu_{\mathrm{prior}}(\mathrm d\sigma).
\end{align*}
The inner integral defines the marginal likelihood (see~\eqref{eq:marginalizeInit}) of the current estimation interval
\begin{align*}
\rho_{\rm like}(\bsz\mid \sigma)
\;:=\;
\int_{\bbR^n}
\rho_{\mathrm{like}}(\bsz\mid\sigma,y_0)\,
\mu_{\mathrm{prior}}(\mathrm dy_0\mid\sigma),
\end{align*}
so that the marginal posterior of~$\sigma$ has Lebesgue density
\begin{align}\label{eq:marg_post}
\rho_{\rm post}(\sigma\mid \bsz)
\;=\;
\frac{1}{Z(\bsz)}\,\rho_{\rm like}(\bsz\mid \sigma)\,\rho_{\rm prior}(\sigma),
\quad
Z(\bsz)=\int_{\mathbb R^p}\rho_{\rm like}(\bsz\mid \sigma)\,\rho_{\rm prior}(\sigma)\,\mathrm d\sigma.
\end{align}
Hence, for the inference of the posterior, the marginal likelihood $\rho_{\rm like}$ is required, which in turn requires the conditional prior on the initial condition~$y_0\mid \sigma$. While the prior on~$y_0 \mid \sigma$ is a Gaussian measure on~${\bbR^n}$
\begin{align}\label{eq:Gauss_assump}
y_0\mid \sigma \sim \mathcal N\big(m_0(\sigma),\mathcal C_0(\sigma)\big),
\end{align}
the likelihood~$\rho_{\rm like}$ depends crucially on the forward model, as explained in the subsequent subsection.

\begin{remark}
    At the initialization step, see \Cref{fig:init_est}, we assume that $y_0^{(1)}\sim\mathcal{N}(m_0^{(1)},\mathcal{C}_0^{(1)})$. Thus, for the first parameter estimation, it can be assumed that $y_0^{(1)}$ is independent of $\sigma$. In this case, it holds that
    $$
    \mu_{\mathrm{prior}}(\mathrm dy_0\mid\sigma)\,\mu_{\mathrm{prior}}(\mathrm d\sigma) \,=\,\mu_{\mathrm{prior}}(\mathrm d\sigma)\,\mu_{\rm prior}(\mathrm dy_0).
    $$
    For the $k$-th parameter estimation ($k\ge 2$), the distribution of the initial condition $y_0^{(k)}\mid \sigma$ is a Gaussian determined by the Kalman filter (see \Cref{sec:Kalmanfilter} below), and does depend on the unknown $\sigma$, see~\eqref{eq:Gaussian_mix2}.
\end{remark}

\subsection{The forward mapping}\label{sec:forwardmap}
The representation~\eqref{eq:timeDiscSys_diffsigma1} together with~\eqref{eq:timeDiscSys_diffsigma2} enables us to define the coupled state
\begin{align*}
    x_j := 
    \begin{bmatrix}
        y_j \\
        \widehat y_j
    \end{bmatrix},\quad 0\le j\le s,
\end{align*}
such that~\eqref{eq:timeDiscSys_diffsigma} can be reformulated as the linear time-discrete system
\begin{align}
    x_{j+1}
    &= A_{\mathrm{cpl}}(\sigma,\widehat{\sigma})\,x_j
       + B_{\mathrm{cpl}}(\widehat{\sigma})\,\eta_j, 
       \label{eq:x_aug_disc}
\\[3pt]
    z_j &= C_{\mathrm{cpl}}\,x_j + \eta_j, \label{eq:z_aug_disc}
\end{align}
where $x_0 =\begin{bmatrix}
    y_0^\top & \widehat{y}_0^\top
\end{bmatrix}^\top$ and
\begin{align*}
    A_{\mathrm{cpl}}(\sigma,\widehat{\sigma})
    =
    \begin{bmatrix}
        \clA_\sigma & \clB_\sigma K_{\widehat{\sigma}} \\[3pt]
        -L_{\widehat{\sigma}} C & \clA_{\widehat{\sigma}} + \clB_{\widehat{\sigma}} K_{\widehat{\sigma}} + L_{\widehat{\sigma}} C
    \end{bmatrix}, \qquad
    B_{\mathrm{cpl}}(\widehat{\sigma})
    = 
    \begin{bmatrix}
        0 \\[3pt]
         -L_{\widehat{\sigma}} 
    \end{bmatrix}, 
    \qquad
    C_{\mathrm{cpl}} = 
    \begin{bmatrix}
        C & 0
    \end{bmatrix}.
\end{align*}
Iterating~\eqref{eq:x_aug_disc} yields that
\begin{align*}
    x_j
    = A_{\mathrm{cpl}}^j(\sigma,\widehat{\sigma})\,x_0
      + \sum_{r=0}^{j-1} 
        A_{\mathrm{cpl}}^{\,j-1-r}(\sigma,\widehat{\sigma})\,
        B_{\mathrm{cpl}}(\widehat{\sigma})\,
        \eta_r.
\end{align*}
Substituting into~\eqref{eq:z_aug_disc} and stacking the observations
$\bsz = \begin{bmatrix}
    z_0^\top,\dots,z_s^\top
\end{bmatrix}^\top$ and noises
$\boldsymbol{\eta} = \begin{bmatrix}
    \eta_0^\top,\dots,\eta_s^\top
\end{bmatrix}^\top$
yields that
\begin{align}
    \bsz
    &= G_x(\sigma,\widehat{\sigma})\,x_0 
       + G_{\eta}(\sigma,\widehat{\sigma})\,\boldsymbol{\eta},
       \label{eq:z_stack}
\end{align}
where
\begin{align*}
    G_x(\sigma,\widehat{\sigma})
    &=
    \begin{bmatrix}
        C_{\mathrm{cpl}}\\
        C_{\mathrm{cpl}} A_{\mathrm{cpl}}(\sigma,\widehat{\sigma}) \\
        C_{\mathrm{cpl}} A_{\mathrm{cpl}}^2(\sigma,\widehat{\sigma}) \\
        \vdots \\
        C_{\mathrm{cpl}} A_{\mathrm{cpl}}^{\,s}(\sigma,\widehat{\sigma})
    \end{bmatrix} \in \mathbb{R}^{\ell(s+1) \times 2n},\end{align*}
\begin{align*}    G_{\eta}(\sigma,\widehat{\sigma})
    &=
    \begin{bmatrix}
    {\rm Id}_{\ell} & 0 & 0 & \cdots & 0\\
    C_{\mathrm{cpl}}B_{\mathrm{cpl}}(\widehat{\sigma}) & {\rm Id}_{\ell} & 0 & \cdots & 0  \\
    C_{\mathrm{cpl}}A_{\mathrm{cpl}}(\sigma,\widehat{\sigma})B_{\mathrm{cpl}} & C_{\mathrm{cpl}}B_{\mathrm{cpl}}(\widehat{\sigma}) & {\rm Id}_{\ell} & \ddots & \vdots  \\
    \vdots & \vdots & \ddots & \ddots & 0  \\
    C_{\mathrm{cpl}}A_{\mathrm{cpl}}^{\,s-1}(\sigma,\widehat{\sigma})B_{\mathrm{cpl}}(\widehat{\sigma}) & C_{\mathrm{cpl}}A_{\mathrm{cpl}}^{\,s-2}(\sigma,\widehat{\sigma})B_{\mathrm{cpl}}(\widehat{\sigma}) & \cdots & C_{\mathrm{cpl}}B_{\mathrm{cpl}}(\widehat{\sigma}) & {\rm Id}_{\ell} 
    \end{bmatrix},
\end{align*}%
where $G_\eta(\sigma,\widehat{\sigma}) \in \mathbb{R}^{\ell(s+1) \times \ell(s+1)}$ and ${\rm Id}_\ell$ denotes the identity matrix in~$\bbR^\ell$.

Recall that in this section, we consider a single parameter estimation, for which $\widehat{\sigma}$ is fixed. Furthermore, when using the data $\bsz = \left[z_0{^\top},\ldots,z_s{^\top}\right]^\top$, the virtual twin trajectory at the time points $t_0,\ldots,t_j$ is determined. In particular, $\widehat{y}_0$ is known. So, according to~\eqref{eq:Gauss_assump}, given $\sigma$, the coupled initial state is Gaussian 
$$x_0 =\begin{bmatrix}y_{0} \\ \widehat{y}_{0} \end{bmatrix} \sim \mathcal{N}\left(\begin{bmatrix}m_0(\sigma) \\ \widehat{y}_0\end{bmatrix}, \begin{bmatrix} \mathcal{C}_0(\sigma) & 0 \\ 0 & 0 \end{bmatrix}\right).$$ 
By linearity, we conclude that
\begin{align*}
    G_x(\sigma,\widehat{\sigma}) x_0 &\sim \mathcal{N}\left(G_x(\sigma,\widehat{\sigma})\begin{bmatrix} m_0(\sigma) \\  \widehat{y}_0\end{bmatrix}, G_x(\sigma,\widehat{\sigma})\begin{bmatrix}  \mathcal{C}_0(\sigma)   & 0 \\ 0 & 0 \end{bmatrix}G_x(\sigma,\widehat{\sigma})^\top\right),\\
    G_{\eta}(\sigma,\widehat{\sigma}) \bseta &\sim \mathcal{N} \left(0 , G_{\eta}(\sigma,\widehat{\sigma}) \bsGammas G_{\eta}(\sigma,\widehat{\sigma})^\top \right),
\end{align*}
where $\bsGammas$ is the block diagonal matrix containing the noise covariances $\Gamma_0,\ldots,\Gamma_s$. From~\eqref{eq:z_stack}, we conclude that 
$$
\bsz \mid \sigma \sim \mathcal{N}\left( \bar \bsz_{\rm cpl}(\sigma,\widehat{\sigma}), \boldsymbol{\Gamma}_{\eta}(\sigma,\widehat{\sigma})\right),
$$  
where
\begin{align*}
\bar \bsz_{\rm cpl}(\sigma,\widehat{\sigma}) &:= G_x(\sigma,\widehat{\sigma})\begin{bmatrix} m_0(\sigma) \\  \widehat{y}_0\end{bmatrix},\\
\boldsymbol{\Gamma}_{\eta}(\sigma,\widehat{\sigma}) &:= G_x(\sigma,\widehat{\sigma})\begin{bmatrix}  \mathcal{C}_0(\sigma)   & 0 \\ 0 & 0 \end{bmatrix}G_x(\sigma,\widehat{\sigma})^\top + G_{\eta}(\sigma,\widehat{\sigma}) \bsGammas G_{\eta}(\sigma,\widehat{\sigma})^\top.
\end{align*}
Then, the marginal likelihood of the data~$\bsz$ given~$\sigma$ is the Gaussian density
\begin{align}\label{eq:rholikenew}
\rho_{\mathrm{like}}(\bsz \mid\sigma)
= \frac{\exp\!\Big(-\tfrac12\big(\bsz-\bar \bsz_{\rm cpl}(\sigma,\widehat{\sigma})\big)^\top
\boldsymbol{\Gamma}_{\eta}(\sigma,\widehat{\sigma})^{-1}\big(\bsz-\bar \bsz_{\rm cpl}(\sigma,\widehat{\sigma})\big)\Big)}{\sqrt{(2\pi)^{-s\ell}\,\det\!\big(\boldsymbol{\Gamma}_{\eta}(\sigma,\widehat{\sigma})\big)}}.
\end{align}

In the $k$-th parameter estimation step, this marginal likelihood $\rho_{\rm like}$ is the likelihood $\rho_{\rm like}^{(k)}(\bsz^{(k)} \mid \bsz^{(1)},\ldots,\bsz^{(k-1)},\sigma)$ in \eqref{eq:posterior-density}, which is used to sample from the posterior distribution of the unknown parameter in the $k$-th parameter estimation problem. This corresponds to step 8 in \Cref{Alg:Offl}.

\subsection{Parameter estimation: two step illustration}\label{sec:para_est2}
To illustrate the parameter estimation procedure introduced in this section, we present the method for two subsequent parameter estimation steps choosing $s+1=6$ observations in each estimation interval.
\begin{figure}[H]
\centering
\vspace{4.95cm}
\begin{minipage}{\linewidth}
\begin{tikzpicture}
  \begin{axis}[
    xmin=-2, xmax=12,
    ymin=-17, ymax=15,
    samples=200,
    axis x line=none,
    axis y line=none,
    width=\textwidth,
    height=6cm,
    scale only axis,
    clip = false
  ]
    \addplot[very thick,black,domain=0:10] {0};
    
    \addplot[only marks, mark=*, mark size=3pt, ultra thick] coordinates {(0,0)};    
    \addplot[only marks, mark=*, mark size=3pt, ultra thick] coordinates {(5,0)};
    \addplot[only marks, mark=*, mark size=3pt, ultra thick] coordinates {(10,0)};
    \node[anchor=north] at (axis cs:0,-11) {\footnotesize$y_0^{(1)}\sim \mathcal{N}(m_0^{(1)},\mathcal{C}_0^{(1)})$};
    \node[anchor=north] at (axis cs:5,-11) {\footnotesize$y_0^{(2)}\sim \mathcal{N}(m_0^{(2)},\mathcal{C}_0^{(2)})$};
    \node[anchor=north] at (axis cs:10,-11) {\footnotesize$y_0^{(3)}\sim \mathcal{N}(m_0^{(3)},\mathcal{C}_0^{(3)})$};

    \node[anchor=north] at (axis cs:0,-7.5) {\footnotesize${\sigma}^{(0)}\sim \mu_{\rm prior}(\mathrm d\sigma)$};
    \node[anchor=north] at (axis cs:5,-7.5) {\footnotesize${\sigma}^{(1)}\sim \mu_{\rm post}^{(1)}(\mathrm d\sigma\mid \bsz^{(1)})$};
    \node[anchor=north] at (axis cs:10,-7.5) {\footnotesize${\sigma}^{(2)}\sim \mu_{\rm post}^{(2)}(\mathrm d\sigma\mid \bsz^{(2)})$};

    \node[anchor=north] at (axis cs:0,-4) {\footnotesize$\widehat{\sigma}^{(0)} = \mathbb{E}_{\mu_{\rm prior}}[\sigma]$};
    \node[anchor=north] at (axis cs:5,-4) {\footnotesize$\widehat{\sigma}^{(1)} = \mathbb{E}_{\mu_{\rm post}^{(1)}}[\sigma]$};
    \node[anchor=north] at (axis cs:10,-4) {\footnotesize$\widehat{\sigma}^{(2)} = \mathbb{E}_{\mu_{\rm post}^{(2)}}[\sigma]$};

    \foreach \r in {1,2,3,4,6,7,8,9} {
    \addplot[only marks, mark=|, mark size=3pt, thick] coordinates {(\r,0)};  
    }

    \node[anchor=north] at (axis cs:0,-15) {\footnotesize$\widehat{y}_0^{(1)} = m_0^{(1)}$};
    \node[anchor=north] at (axis cs:5,-15) {\footnotesize$\widehat{y}_0^{(2)} = \widehat{y}_5^{(1)}$};
    \node[anchor=north] at (axis cs:10,-15) {\footnotesize$\widehat{y}_0^{(3)} = \widehat{y}_5^{(2)}$};

    \node[anchor=north] at (axis cs:0,-.5) {\footnotesize$y_0^{(1)}$};    
    \node[anchor=north] at (axis cs:5,-.5) {\footnotesize$y_0^{(2)}$};
    \node[anchor=north] at (axis cs:10,-.5) {\footnotesize$y_0^{(3)}$};

    \node[anchor=south] at (axis cs:0,.3) {\footnotesize$z_0^{(1)}$};
    \node[anchor=south] at (axis cs:1,.3) {\footnotesize$z_1^{(1)}$};
    \node[anchor=south] at (axis cs:2,.3) {\footnotesize$z_2^{(1)}$};
    \node[anchor=south] at (axis cs:3,.3) {\footnotesize$z_3^{(1)}$};
    \node[anchor=south] at (axis cs:4,.3) {\footnotesize$z_4^{(1)}$};
    \node[anchor=south] at (axis cs:5,.3) {\footnotesize$z_5^{(1)}\!\!=\!z_0^{(2)}$};

    \node[anchor=south] at (axis cs:6,.3) {\footnotesize$z_1^{(2)}$};
    \node[anchor=south] at (axis cs:7,.3) {\footnotesize$z_2^{(2)}$};
    \node[anchor=south] at (axis cs:8,.3) {\footnotesize$z_3^{(2)}$};
    \node[anchor=south] at (axis cs:9,.3) {\footnotesize$z_4^{(2)}$};
    \node[anchor=south] at (axis cs:10,.3) {\footnotesize$z_5^{(2)}\!\!=\!z_0^{(3)}$};

    \node[anchor=south] at (axis cs:1.5,5) {\footnotesize $\begin{bmatrix}y_{j+1}^{(1)} \\ \widehat{y}_{j+1}^{(1)} \end{bmatrix} = A_{\rm cpl}(\sigma,\widehat{\sigma}^{(0)}) \begin{bmatrix}y_{j}^{(1)} \\ \widehat{y}_{j}^{(1)} \end{bmatrix} + B_{\rm cpl}(\widehat{\sigma}^{(0)}) \eta_j^{(1)}
    $};

    \node[anchor=south] at (axis cs:8.5,5) {\footnotesize $\begin{bmatrix}y_{j+1}^{(2)} \\ \widehat{y}_{j+1}^{(2)} \end{bmatrix} = A_{\rm cpl}(\sigma,\widehat{\sigma}^{(1)}) \begin{bmatrix}y_{j}^{(2)} \\ \widehat{y}_{j}^{(2)} \end{bmatrix} + B_{\rm cpl}(\widehat{\sigma}^{(1)}) \eta_j^{(2)}
    $};
  \end{axis}
\end{tikzpicture}
\end{minipage}
\vspace{-0.5cm}
\caption{Illustration for $k=1,2$.}
\label{fig:init_est}
\end{figure}
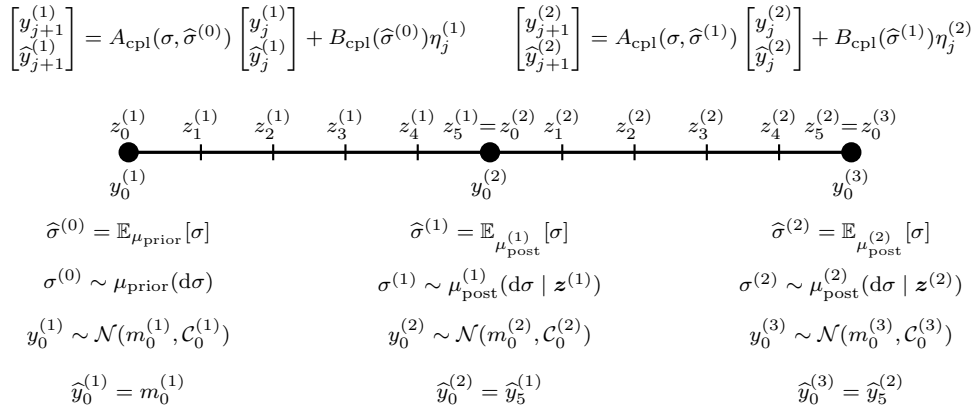

Our strategy is as follows:
\begin{itemize}[wide=0pt]
    \item[\underline{Initialization.}] We are at the most left point in \Cref{fig:init_est}. We are given $\mu_{\rm prior}(\mathrm d\sigma)$ and $y_0^{(1)} \sim \mathcal{N}(m_0^{(1)},\mathcal{C}_0^{(1)})$ by assumption. We initialize $\widehat{y}_0^{(1)} = m_0^{(1)}$ and $\widehat{\sigma}^{(0)} = \mathbb{E}_{\mu_{\rm prior}}[\sigma]$ and run the coupled system characterized by $(A_{\rm cpl}(\sigma,\widehat{\sigma}^{(0)}),B_{\rm cpl}(\widehat{\sigma}^{(0)}))$ for $s=5$ steps. While the system is running, we collect the performed measurements $\bsz^{(1)} = \left[z_0^{(1)\top},\ldots,z_5^{(1)\top}\right]^\top$.
    \item[\underline{1st estimation.}] We are at the center in \Cref{fig:init_est}. We are given $y_0^{(1)} \sim \mathcal{N}(m_0^{(1)},\mathcal{C}_0^{(1)})$ and $\mu_{\rm prior}(\mathrm d\sigma)$. Using this, we can sample from the posterior $\sigma^{(1)} \sim \mu_{\rm post}^{(1)}(\mathrm d\sigma \mid \bsz^{(1)})$ (marginalized over $y_0^{(1)}$) as detailed in \Cref{sec:singleBayes} (and in \Cref{sec:para_est_general} later on). We set
    $$    \widehat{\sigma}^{(1)} = \mathbb{E}_{\mu_{\rm post}^{(1)}}[\sigma].
    $$
    At this point, $\widehat{y}_j^{(1)}$ are known for $j \in \{0,\ldots,5\}$. They will be used to determine the distribution of the initial condition $y_0^{(2)}$ of the subsequent interval. We update $(A_{\rm cpl}(\sigma,\widehat{\sigma}^{(1)}),B_{\rm cpl}(\widehat{\sigma}^{(1)}))$ and run the coupled system $s=5$ additional steps. While the system is running, we  collect measurements $\bsz^{(2)} = \left[z_0^{(2)\top},\ldots,z_5^{(2)\top} \right]^\top$.
    \item[\underline{2nd estimation.}] We are at the most right point in \Cref{fig:init_est}. As before, updating the parameter
    $$
    \widehat{\sigma}^{(2)} = \mathbb{E}_{\mu_{\rm post}^{(2)}}[\sigma],
    $$
    requires samples from the posterior $\mu_{\rm post}^{(2)}(\mathrm d\sigma\mid \bsz^{(2)})$. For this purpose, we use the posterior from the previous iteration $\mu_{\rm post}^{(1)}(\mathrm d\sigma\mid \bsz^{(1)})$ as prior and marginalize over the initial condition $y_0^{(2)}$. To this end, we use Kalman filtering to obtain the conditional distribution $\mu(\mathrm dy_0^{(2)} \mid \bsz^{(1)}, \sigma)$ of $y_0^{(2)}$ given $\sigma$, and the data that is available at this point $\bsz^{(1)}$, i.e., the Kalman filter propagates $m_0^{(k)}$ and $\mathcal{C}_0^{(k)}$. Below, in \Cref{sec:Kalmanfilter}, we provide more details on the Kalman filter. At this point, the virtual twin trajectory is determined completely.

    For more than two estimation intervals, the exact same strategy is repeated: we update the matrices $(A_{\rm cpl}(\sigma,\widehat{\sigma}^{(2)}),B_{\rm cpl}(\widehat{\sigma}^{(2)}))$ and run the coupled system $s=5$ additional steps. While the system is running, we collect measurements $\bsz^{(3)} = \left[z_0^{(3)\top},\ldots,z_5^{(3)\top}\right]^\top$, which we will use in the next parameter estimation, and repeat the steps described above. The repeated application of this method for an arbitrary number of estimation intervals is described in the subsequent subsection.
\end{itemize}

\subsection{Parameter estimation: general case}\label{sec:para_est_general}
Suppose we are at time $t_j$, $j = sk$ for $k\ge 2$ (for $k=2$, this is the most right point in \Cref{fig:init_est}). At this point, we wish to estimate $\widehat{\sigma}^{(k)}$, and we are given a distribution~$\mu_{\rm post}^{(k-1)}(\mathrm d\sigma \mid \bsz^{(k-1)})$ from the previous estimation interval, all measurements $\bsz^{(i)}$, $i \le k$, and all values of the virtual twin up to this point. Let us denote the virtual twin evolution in the $k$-th interval by $\widehat{\boldsymbol{y}}^{(k)} = (\widehat{y}_0,\ldots,\widehat{y}_{s-1})$ and let us point out that the virtual twin is a deterministic function of the observations, see~\eqref{eq:timeDiscSys_diffsigma2}. Thus, given $\bsz^{(k-1)}$ and the initial condition~$\widehat{y}_0$, the virtual twin $\widehat{\bsy}^{(k-1)}$ is fully determined. Conditioning on the virtual twin $\widehat{\bsy}^{(1)},\ldots,\widehat{\bsy}^{(k-1)}$ in addition to~$\bsz^{(1)},\ldots,\bsz^{(k-1)}$ provides no additional information compared to conditioning on~$\bsz^{(1)},\ldots,\bsz^{(k-1)}$ only, that is
$$
\rho^{(k)}_{\rm like}(\bsz^{(k)} \mid \bsz^{(1)},\ldots,\bsz^{(k-1)}, \sigma, \widehat{\bsy}^{(1)},\ldots,\widehat{\bsy}^{(k-1)}) = \rho^{(k)}_{\rm like}(\bsz^{(k)} \mid \bsz^{(1)},\ldots,\bsz^{(k-1)}, \sigma).
$$
Therefore, conditioning on~$\widehat{\bsy}^{(1)},\ldots,\widehat{\bsy}^{(k-1)}$ will be omitted in the following.

By the law of total probability we have that
\begin{align*}
&\rho^{(k)}_{\rm like}(\bsz^{(k)} \mid \bsz^{(1)},\ldots,\bsz^{(k-1)}, \sigma)
\\\quad&=
\int_{\bbR^n} 
\rho(\bsz^{(k)} \mid y_{0}^{(k)}, \bsz^{(1)},\ldots,\bsz^{(k-1)}, \sigma)\,\rho_{\rm Kal}(y_{0}^{(k)} \mid \bsz^{(1)},\ldots,\bsz^{(k-1)}, \sigma)\,
\mathrm dy_{0}^{(k)}.
\end{align*}
The second term, $\rho_{\rm Kal}(y_{0}^{(k)} \mid \bsz^{(1)},\ldots,\bsz^{(k-1)},\sigma)$, is obtained by the Kalman filter as presented in \Cref{sec:Kalmanfilter} below. According to the measurement model, once the state $y_j$ and the unknown parameter $\sigma$ are known, the observation $z_j$ does not depend on past states, the virtual twin trajectory, and past observations. Hence, the first term in the likelihood simplifies to
$$
\rho^{(k)}(\bsz^{(k)} \mid y_{0}^{(k)}, \bsz^{(1)},\ldots,\bsz^{(k-1)}, \sigma)
=
\rho(\bsz^{(k)} \mid y_{0}^{(k)}, \sigma),
$$
which yields that
\begin{align}\label{eq:marginalizeInit}
\rho_{\rm like}^{(k)}(\bsz^{(k)} \mid \bsz^{(1)},\ldots,\bsz^{(k-1)}, \sigma)=
\int_{\bbR^n}
\rho(\bsz^{(k)} \mid y_{0}^{(k)}, \sigma)\,
\rho_{\rm Kal}(y_{0}^{(k)} \mid \bsz^{(1)},\ldots,\bsz^{(k-1)}, \sigma )\,
\mathrm dy_{0}^{(k)}.
\end{align}

Using the likelihood density of future observations conditioned on past data and the unknown parameter
$$
\rho_{\rm like}^{(k)}(\bsz^{(k)} \mid \bsz^{(1)},\ldots,\bsz^{(k-1)}, \sigma),
$$
we can construct the posterior density of $\sigma$, and given all observations $\bsz^{(1)},\ldots,\bsz^{(k)}$ using Bayes' theorem
\begin{align}
\label{eq:posterior-density}
\rho_{\rm post}^{(k)}(\sigma \mid \bsz^{(1)},\ldots,\bsz^{(k)}) = \frac{\rho_{\rm like}^{(k)}(\bsz^{(k)} \mid \bsz^{(1)},\ldots,\bsz^{(k-1)}, \sigma) \,
\rho_{\rm post}^{(k-1)}(\sigma \mid \bsz^{(1)},\ldots,\bsz^{(k-1)})}{\rho(\bsz^{(k)} \mid \bsz^{(1)},\ldots,\bsz^{(k-1)})},
\end{align}
where the posterior from the previous step $\rho_{\rm post}^{(k-1)}(\sigma \mid \bsz^{(1)},\ldots,\bsz^{(k-1)})$ serves as prior for the $k$-th update, and $\rho_{\rm post}^{(0)} = \rho_{\rm prior}$.

\section{Estimation of the initial state distribution using Kalman-filtered Gaussians}\label{sec:Kalmanfilter}

\subsection{Kalman filtering on multiple intervals}
With reference to \Cref{fig:init_est}, we refer to the intervals between the big bullets as estimation intervals. Each of these is divided into $s$ time steps at which data are collected. For the parameter estimation at time $t_j$ with $j = (k-1)s$ for $k\ge 2$, a distribution of the unknown initial condition of the physical system $y_0^{(k)}$ given all available data $\bsz^{(1)},\ldots,\bsz^{(k-1)}$ and the unknown parameter $\sigma$ is required, see~\eqref{eq:marginalizeInit}. For the estimation of the first $\widehat{\sigma}^{(k)}$, i.e., with $k=1$, we assume that $y_0 = y_0^{(1)} \sim \mathcal{N}(m_0^{(1)},\mathcal{C}_0^{(1)})$. For the subsequent estimations of $\widehat{\sigma}^{(k)}$, we utilize the distribution
\begin{align}\label{eq:Gaussian_mix2}
\mu(\mathrm dy_{(k-1)s} \mid \bsz^{(1)},\ldots, \bsz^{(k-1)}, \sigma) = \mathcal{N}(m_0^{(k)}(\sigma), \mathcal{C}_0^{(k)}(\sigma)),
\end{align}
with mean $m_0^{(k)}(\sigma) = 
m_{y\mid \widehat{y};s}^{(k-1)}(\sigma)$ and covariance $\mathcal{C}_0^{(k)}(\sigma) = C_{y \mid \widehat{y};s}^{(k-1)}(\sigma)$ to be defined in \Cref{sec:KalmanFilter}.
In this subsection and in \Cref{sec:KalmanFilter} below we explain how this distribution can be computed using the Kalman filter on every parameter estimation interval.
We denote the Lebesgue density of $\mu(\mathrm dy_{(k-1)s}\mid \bsz^{(1)},\ldots, \bsz^{(k-1)}, \sigma)$ by $\rho_{\rm Kal}(y_0^{(k)} \mid \bsz^{(1)},\ldots,\bsz^{(k-1)},\sigma)$, see~\eqref{eq:marginalizeInit}.

By the law of total probability, conditioning on the previous interval's initial state $y_0^{(k-1)}$ gives
\begin{align}\label{eq:multi_int_Kalman}
&\rho_{\rm Kal}\big(y_0^{(k)} \mid \bsz^{(1)},\ldots,\bsz^{(k-1)}, \sigma \big)\\
&\quad= \int_{\bbR^n} \rho\big(y_0^{(k)} \mid y_0^{(k-1)}, \bsz^{(1)},\ldots,\bsz^{(k-1)}, \sigma \big)
       \rho\big(y_0^{(k-1)} \mid \bsz^{(1)},\ldots,\bsz^{(k-1)}, \sigma \big) \, \mathrm{d}y_0^{(k-1)}. \notag
\end{align}

In the multi-interval setting with $y_0^{(k)} = y_{(k-1)s}$, we have the Markov property
$$
y_0^{(k)} \;\perp\; \bsz^{(1)},\ldots,\bsz^{(k-2)} \mid y_0^{(k-1)}, \bsz^{(k-1)}, \sigma,
$$
i.e., the initial state of the $k$-th interval depends on all past measurements only through the final state and measurements of the previous interval. Hence, for the first factor on the right-hand side of~\eqref{eq:multi_int_Kalman}, it holds that
$$
\rho\big(y_0^{(k)} \mid y_0^{(k-1)}, \bsz^{(1)},\ldots,\bsz^{(k-1)}, \sigma \big)
=
\rho\big(y_0^{(k)} \mid y_0^{(k-1)}, \bsz^{(k-1)}, \sigma \big).
$$
Regarding the second factor on the right-hand side of~\eqref{eq:multi_int_Kalman}, given $\sigma$ and $\bsz^{(k-2)}$, the initial condition $y_0^{(k-1)}$ is clearly independent of the future measurements $\bsz^{(k-1)}$, so 
$$
\rho\big(y_0^{(k-1)} \mid \bsz^{(1)},\ldots,\bsz^{(k-1)}, \sigma \big)
=
\rho\big(y_0^{(k-1)} \mid \bsz^{(1)},\ldots,\bsz^{(k-2)}, \sigma \big).
$$
This yields the recursion
\begin{align*}
&\rho_{\rm Kal}\big(y_0^{(k)} \mid \bsz^{(1)},\ldots,\bsz^{(k-1)}, \sigma \big)\\
=
&\quad\int_{\bbR^n} 
\rho\big(y_0^{(k)} \mid y_0^{(k-1)}, \bsz^{(k-1)}, \sigma \big)
\rho\big(y_0^{(k-1)} \mid \bsz^{(1)},\ldots,\bsz^{(k-2)}, \sigma \big)
\, \mathrm{d}y_0^{(k-1)}.
\end{align*}
Solving this recursion until the first estimation interval with $y_0^{(1)} = y_0$ gives
\begin{equation*}
\begin{aligned}
&\rho_{\rm Kal}\big(y_0^{(k)} \mid \bsz^{(1)},\ldots,\bsz^{(k-1)}, \sigma \big)\\
&\qquad= \int_{\bbR^{n \times(k-1)}} \left(\prod_{i=1}^{k-1} \rho(y_0^{(i+1)} \mid y_0^i,\bsz^{(i)},\sigma) \right) \rho(y_0^{(1)})\, \mathrm d y_0^{(1)}\cdots\mathrm d y_0^{(k-1)}.
\end{aligned}
\end{equation*}
For each $k$, this conditional density can be computed exactly by means of the Kalman filter. In particular, the Kalman filter admits a natural extension across multiple estimation intervals: the density over all $k$ estimation intervals defined by the above recursion is obtained by successive applications of the Kalman filter, using the posterior from the previous estimation interval as the initial condition for the next one. Accordingly, in the following section we present the Kalman filter formulation for a single estimation interval.

\subsection{Kalman filtering on a single estimation interval}\label{sec:KalmanFilter}
The aim of this section is to compute the distribution~\eqref{eq:Gaussian_mix2} for a single estimation interval. We will use the Kalman filter to obtain the distribution of the coupled virtual-physical twin state and then condition on the realization of the virtual twin. This is equivalent to interpreting the virtual twin as an additional noise-free observation and applying the Kalman filter to this augmented observation model.

Given $\sigma$, used in $\mathcal{A}_{\sigma}$, the coupled system after the $k$-th parameter update is linear and Gaussian. Precisely, it is 
\begin{align*}
\begin{bmatrix}y_{j+1}^{(k)} \\ \widehat{y}_{j+1}^{(k)} \end{bmatrix} &= A_{\rm cpl}(\sigma,\widehat{\sigma}^{(k)}) \begin{bmatrix}y_{j}^{(k)} \\ \widehat{y}_{j}^{(k)} \end{bmatrix} + B_{\rm cpl}(\widehat{\sigma}^{(k)}) \eta_j,  &&j = 0,\ldots, s-1,\\ \begin{bmatrix}y_{0}^{(k)} \\ \widehat{y}_{0}^{(k)} \end{bmatrix} &\sim \mathcal{N}\left(\begin{bmatrix}m_0^{(k)}(\sigma) \\ \widehat{y}_0^{(k)}\end{bmatrix}, \begin{bmatrix} \mathcal{C}_0^{(k)}(\sigma) & 0 \\ 0 & 0 \end{bmatrix}\right),\\
z_j^{(k)} &= C_{\rm cpl} \begin{bmatrix}y_{j}^{(k)} \\ \widehat{y}_{j}^{(k)} \end{bmatrix} + \eta_j,  &&j = 0,\ldots,s,
\end{align*}
where $A_{\rm cpl}$, $B_{\rm cpl}$, and $C_{\rm cpl}$ are as in \Cref{sec:forwardmap}. Recall that~$\widehat{y}_0^{(k)}$, the initial condition of the virtual twin, is known when we estimate the unknown parameter $\sigma$. This is modeled by the covariance identical to zero above.

For fixed~$\widehat{\sigma}$, the conditional distribution of the pair $x_j^{(k)} = \begin{bmatrix} y_{j}^{(k)} & \widehat{y}_{j}^{(k)} \end{bmatrix}$, given the data $z_{j}^{(k)}$ and the parameter $\sigma$, is again Gaussian. It is determined by the Kalman filter in terms of the mean and covariance. The Kalman filter reads \cite[Chapter 5]{Simon06}:
\paragraph{Initialization:}
\begin{align*}
    m_{0|0}^{(k)}(\sigma) = \begin{bmatrix}m_0^{(k)}(\sigma) \\ \widehat{y}_0^{(k)}\end{bmatrix},\quad \text{and} \quad
    \mathcal{C}_{0|0}^{(k)}(\sigma) = \begin{bmatrix} \mathcal{C}_0^{(k)}(\sigma) & 0 \\ 0 & 0 \end{bmatrix}.
\end{align*}

Then, for $j = 1,\ldots, s$:
\paragraph{Prediction:}
\begin{align*}
m_{j|j-1}^{(k)}(\sigma) &= A_{\rm cpl}(\sigma,\widehat{\sigma}^{(k)}) m_{j-1|j-1}^{(k)}(\sigma),\\
\mathcal{C}_{j|j-1}^{(k)}(\sigma) &= A_{\rm cpl}(\sigma,\widehat{\sigma}^{(k)}) \mathcal{C}_{j-1|j-1}^{(k)}(\sigma)A_{\rm cpl}(\sigma,\widehat{\sigma}^{(k)})^\top
                    + B_{\mathrm{cpl}}(\widehat{\sigma}^{(k)}) \Gamma_j B_{\mathrm{cpl}}(\widehat{\sigma}^{(k)})^\top.
\end{align*}
\paragraph{Innovation and Kalman Gain:}
\begin{align*}
\text{innovation: } \nu_j(\sigma) &= z_j^{(k)} - C_{\mathrm{cpl}}\, m_{j|j-1}^{(k)}(\sigma),\\
\text{innovation covariance: } S_j(\sigma) &= C_{\mathrm{cpl}}\, \mathcal{C}_{j|j-1}^{(k)}(\sigma)\, C_{\mathrm{cpl}}^\top + \Gamma_j,\\
\text{Kalman gain: } K_j(\sigma) &= \mathcal{C}_{j|j-1}^{(k)}(\sigma)\, C_{\mathrm{cpl}}^\top\, S_j(\sigma)^{-1}.
\end{align*}
\paragraph{Coupled update:}
\begin{align*}
\mathfrak{m}_{j|j}(\sigma) &= m_{j|j-1}^{(k)}(\sigma) + K_j(\sigma)\, \nu_j(\sigma),\\
\mathfrak{C}_{j|j}(\sigma) &= \mathcal{C}_{j|j-1}^{(k)}(\sigma) - K_j(\sigma)\, S_j(\sigma)\, K_j(\sigma)^\top.
\end{align*}

\paragraph{Conditioning on the virtual twin:}

\begin{align*}
m_{y\mid \widehat y;j}^{(k)} &= [\mathfrak{m}_{j|j}(\sigma)]_{1:d} - [\mathfrak{C}_{j|j}(\sigma)]_{d+1:2d,1:d}[\mathfrak{C}_{j|j}(\sigma)]_{d+1:2d,d+1:2d}^{-}\left(\widehat{y}_j^{(k)}-[\mathfrak{m}_{j|j}(\sigma)]_{d+1:2d}\right),\\
C_{y \mid \widehat{y};j}^{(k)} &= [\mathfrak{C}_{j|j}(\sigma)]_{1:d,1:d} - [\mathfrak{C}_{j|j}(\sigma)]_{1:d,d+1:2d}[\mathfrak{C}_{j|j}(\sigma)]_{d+1:2d,d+1:2d}^{-}[\mathfrak{C}_{j|j}(\sigma)]_{d+1:2d,1:d},
\end{align*}
where $[v]_{n:l}$ denotes a vector containing components $v_{n},v_{n+1},\ldots,v_{l}$ of a vector $v$, and $[A]_{n:l,p:q}$ denotes the submatrix containing the rows $n, n+1,\ldots,l$ and columns $p,p+1,\ldots,q$ of a matrix $A$. Here, $[\mathfrak{C}_{j|j}(\sigma)]_{d+1:2d,d+1:2d}^{-}$ denotes the generalized inverse of $[\mathfrak{C}_{j|j}(\sigma)]_{d+1:2d,d+1:2d}$, see~\cite[Prop. 3.13]{eaton2007multivariate}.

\paragraph{Conditioned update:}
\begin{align*}
    m_{j\mid j}^{(k)}(\sigma) = \begin{bmatrix}
        m_{y\mid \widehat{y};j}^{(k)} \\ \widehat{y}_j^{(k)}
    \end{bmatrix}\quad \text{and} \quad C_{j \mid j}^{(k)}(\sigma) = \begin{bmatrix}
        \mathcal{C}_{y\mid \widehat{y};j}^{(k)} & 0 \\ 0 & 0 
    \end{bmatrix}.
\end{align*}

In the coupled update step we compute, for fixed $\widehat{\sigma}$ and $1 \le j\le s$, the conditional distribution of the coupled state given the data $(z_0^{(k)},\ldots,z_j^{(k)})$ and given the parameter $\sigma$. Since the virtual twin component $\widehat{y}_j^{(k)}$ is known at the time of the observation $z_j^{(k)}$, we then compute the conditional distribution of the physical state $y_j^{(k)}$ given the virtual twin $\widehat{y}_j^{(k)}$.

In the next estimation interval, we initialize with $m_{0\mid 0}^{(k+1)}(\sigma) = m_{s\mid s}^{(k)}(\sigma)$ and $\mathcal{C}_{0\mid 0}^{(k+1)}(\sigma) = C_{s\mid s}^{(k)}(\sigma)$. This procedure provides us with the distribution of the initial condition~\eqref{eq:Gaussian_mix2} on each estimation interval. It enters the parameter estimation via its density denoted by~$\rho_{\rm Kal}$ in~\eqref{eq:marginalizeInit}.

\section{Numerical experiments}\label{sec:numerics}

Our digital twin strategy is summarized in \Cref{Alg:Offl}. In~\Cref{Alg:Onl}, a detailed description of the computation of the posterior distribution in Step 13 of \Cref{Alg:Offl} is provided.
With reference to \Cref{rem:ricgains}, unless otherwise specified, we set $R=1$ and $Q=10^{-12}\cdot {\rm Id_2}$. Further details on the implementation of the examples listed above are provided with the examples. These are a harmonic oscillator, a spring-damper-system, and a finite element discretization of a parabolic partial differential equation. 
\begin{figure}[htbp]
\centering
\includegraphics[width=0.95\textwidth]{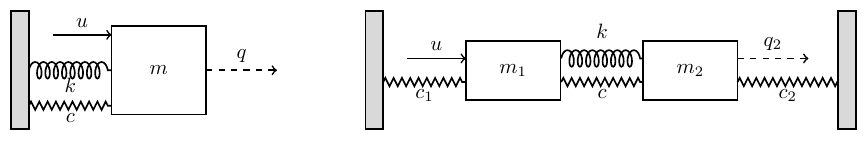}
\caption{Comparison of the single mass oscillator system (left) and a two-mass (right) spring--damper system.}
\label{fig:tikz}
\end{figure}
The oscillator example and the spring-damper-system are illustrated in \Cref{fig:tikz}.

\begin{algorithm}[H]
 \caption{Digital twin stabilization}
\begin{algorithmic}[1]\label{Alg:Offl}
\REQUIRE{Prior distribution of the parameters, prior distribution of the initial condition $\mathcal{N}(m_0^{(1)},\mathcal{C}^{(1)})$, control operator $B$, output operator $C$, initial guess for the digital twin, e.g.,~$\widehat{y}_0^{(1)} = m_0^{(1)}$.}
\ENSURE{Feedback control input $u =K_{\widehat{\sigma}}\widehat{y}$.}

\STATE $k = 0$;
\STATE Compute an estimate~$\widehat{\sigma}^{(k)}$ for~$\sigma \sim \mu_{\rm prior}$, e.g.,~$\widehat{\sigma}^{(k)} = \mathbb{E}_{\mu_{\rm prior}}[\sigma]$;
\STATE Assemble~$A_{\widehat{\sigma}^{(k)}}$ and compute $\mathcal{A}_{\widehat{\sigma}^{(k)}}$ and $\mathcal{B}_{\widehat{\sigma}^{(k)}}$ as in \Cref{re:appr};
\STATE Construct~$K_{\widehat{\sigma}^{(k)}}$ and~$L_{\widehat{\sigma}^{(k)}}$ as in \Cref{rem:ricgains};
\WHILE{the coupled system~\eqref{eq:timeDiscSys_diffsigma} is running}
\STATE $k \leftarrow k+1$;
\FOR{$j=0:s$}
\STATE Collect data~$z_j^{(k)} =z(t_j)$ from the physical model according to~\eqref{eq:obs};
\ENDFOR
\IF{$k>1$}
\STATE Run Kalman filter (see \eqref{eq:Gaussian_mix2}) to obtain~$y_0^{(k)}\sim \mathcal{N}(m_0^{(k)},\mathcal{C}_0^{(k)})$;
\ENDIF
\STATE Update~$\widehat{\sigma}$ based on the posterior~$\mu_{\rm post}^{(k)}$ (see \eqref{eq:posterior-density}), e.g.,~$\widehat{\sigma}^{(k)} = \mathbb{E}_{\mu_{\rm post}^{(k)}}[\sigma]$;
\STATE Construct~$K_{\widehat{\sigma}}^{(k)}$ and~$L_{\widehat{\sigma}}^{(k)}$, assemble~$A_{\widehat{\sigma}}^{(k)}$ and compute $\mathcal{A}_{\widehat{\sigma}^{(k)}}$ and $\mathcal{B}_{\widehat{\sigma}^{(k)}}$;
\ENDWHILE
 \end{algorithmic}
\end{algorithm}

\begin{algorithm}[htb]
\caption{Sequential Monte Carlo with Resample--Move Step}\label{Alg:Onl}
\begin{algorithmic}[1]
\REQUIRE{Prior particles $\Sigma_0 = \{\sigma_i\}_{i=1}^N$, observations $\bsz^{(k)}$ on intervals $\{\mathcal{I}^k_{\rm obs}\}$, number of MCMC steps $n_{\rm MCMC}$, proposal scale $\varepsilon$.}

\ENSURE{Particle approximation $\Sigma_k$ of the posterior at each interval.}

\STATE Initialize sequential likelihoods $\rho_{\rm like}^0(\bsz^{(0)}\mid\sigma_i) = 1$ and particle set $\Sigma_1 = \Sigma_0$;

\FOR{$k = 1,2,\dots$} 

\FOR{$i = 1,\ldots,N$}
    \STATE Compute likelihood increment $\rho_{\rm like}^{\rm new}(\bsz^{(k)}\mid \sigma_i)$ on $\mathcal{I}^k_{\rm obs}$ as in \eqref{eq:rholikenew};
    \STATE Update sequential likelihood:
    \[
        \rho_{\rm like}^{(k)}(\bsz^{(k)} \mid \sigma_i) = \rho_{\rm like}^{(k-1)}(\bsz^{(k-1)}\mid \sigma_i) \cdot \rho_{\rm like}^{\rm new}(\bsz^{(k)}\mid \sigma_i);
    \]

    \STATE Compute normalized weights:
    \[
        w_i \propto \rho_{\rm like}^{(k)}(\bsz^{(k)}\mid\sigma_i), \quad \sum_{i=1}^N w_i = 1;
    \]
\ENDFOR

    \STATE Resample $N$ particles according to $w_1,\dots,w_N$:
    \[
        \Sigma_k = \{\sigma_{\iota_i} \in \Sigma_{k-1} \,:\, i=1,\dots,N\}, 
        \quad \text{with } \Pr(\iota_i = \ell) = w_\ell \,\text{for }\ell = 1,\ldots, N;
    \]

    \FOR{$s=1,\dots,n_{\rm MCMC}$}
        \FOR{$i=1,\dots,N$}
            \STATE Propose candidate $\sigma_i' = \sigma_i + \varepsilon \, \xi$, $\xi \sim \mathcal{N}(0, I)$;
            
            \STATE Evaluate likelihood increment for candidate $\rho_{\rm like}^{\rm new}(\bsz^{(k)}\mid\sigma_i')$ according to \eqref{eq:rholikenew}:

            \STATE Approximate candidate sequential likelihood:
            \[
                \rho_{\rm like}'(\bsz^{(k)}\mid\sigma_i') = \rho_{\rm like}^{(k-1)}(\bsz^{(k-1)}\mid\sigma_i) \cdot \rho_{\rm like}^{\rm new}(\bsz^{(k)}\mid\sigma_i');
            \]
            \STATE Accept candidate with probability
            \[
                \alpha = \min\left(1, \frac{\rho_{\rm like}'(\bsz^{(k)}\mid\sigma_i') \rho_{\rm prior}(\sigma_i')}{\rho_{\rm like}^{(k)}(\bsz^{(k)}\mid\sigma_i)\rho_{\rm prior}(\sigma_i)} \right);
            \]
            \IF{accepted}
                \STATE $\sigma_i = \sigma_i'$;
                \STATE $\rho_{\rm like}^{(k)}(\bsz\mid\sigma_i) = \rho_{\rm like}'(\bsz\mid\sigma_i')$;
            \ENDIF
        \ENDFOR
    \ENDFOR
\ENDFOR

\end{algorithmic}
\end{algorithm}

\subsection{Oscillator}\label{sec:oscnum}
We consider a mass attached to a spring, as depicted in \Cref{fig:tikz} (left).

Let~$x(t)$ be the position of the mass~$m$ and let~$x_{0}$ be an equilibrium position at rest. 
Now, consider the relative position~$q(t)=x(t)-x_{0}$. For the dynamics, we have the oscillator model (for small~$q$)
\begin{align}
    m\ddot q&=-c\dot q-kq+u.
\end{align}
The input control forcing~$u$ acts on the mass~$m$ and we measure the (relative) position~$q$. Hereafter we take~$k=m=1$ and consider the damping constant~$c=\sigma$ as the uncertain parameter. Denoting~$Q=(q,\dot q)$, we write the system as
\begin{subequations}\label{sys-osci-sig}
 \begin{align}
   \dot Q&=A_\sigma Q+Bu,\\
   \text{with}\quad A_\sigma&=\begin{bmatrix}
       0&1\\
       -1&-\sigma
   \end{bmatrix},\qquad
   B=\begin{bmatrix}
       0\\1
   \end{bmatrix},\qquad
     C=\begin{bmatrix}
       1&0
   \end{bmatrix}.\notag
\end{align}
\end{subequations}
The Kalman controllability~$\clK_\clC$ and observability~$\clK_\clO$ matrices are given by
\begin{align*}
  \clK_\clC=\begin{bmatrix}
       0&1\\
       1&-\sigma
   \end{bmatrix},
   \qquad
   \clK_\clO = \begin{bmatrix}
       1&0\\
       0&1
       \end{bmatrix},
\end{align*}
which are both full-rank. In particular,~$(A_\sigma,B)$ is stabilizable  and~$(A_\sigma,C)$ is detectable for all~$\sigma\in\bbR$.

\begin{remark}
 Usually, the damping parameter is positive $\sigma>0$. The damping free case~$\sigma=0$ is relevant as well. We also test our strategy with the academic case~$\sigma<0$, which is of interest because in this case the free dynamics is exponentially unstable and the input will have to be able to counteract this instability.
\end{remark}

The numerical results displayed in~\Cref{fig:osc1} and~\Cref{fig:osc2} are obtained using $N = 5000$ particles initialized from a uniform prior distribution on the interval $[-1,1]$. The true parameter is $\sigma = \tfrac{\pi}{7}\approx 0.4488$ in the stable case and $\sigma = -\tfrac{\pi}{7}\approx -0.4488$ in the unstable case. In the stable case, the estimate $\widehat{\sigma}$ is updated at times $2,4,6,\ldots,20$ with observations at times $1,2,3,\ldots,20$, while in the unstable case, $\widehat{\sigma}$ is updated more frequently at times $1,2,3,\ldots,20$ with more observations at times $0.25,0.5,0.75,\ldots,20$. In both cases, we initialize with $y_0 = \begin{bmatrix} 1 & -1 \end{bmatrix}^\top$ and $\widehat{y}_0 =\begin{bmatrix} 1.5 & 0 \end{bmatrix}^\top $. Further, we assume that the distribution of the initial condition is $\mathcal{N}\Big(\begin{bmatrix} 1.5 & -1.1 \end{bmatrix}^\top, {\rm Id_2}\Big)$ and that the distribution of the noise is $\mathcal{N}(0,0.015\|y_0\|)$. The matrices \(\mathcal{A}_{\sigma},\clB_{\sigma}\) are computed using the matrix exponential as $\clA_\sigma = e^{A_\sigma \Delta t}$ and $\clB_{\sigma} = \int_0^{\Delta t} e^{A_\sigma(\Delta t-\tau)}\,\mathrm d\tau\, B$, whereas the matrices \(\clA_{\widehat{\sigma}},\clB_{\widehat{\sigma}}\) are computed using a Crank--Nicolson discretization as in \Cref{re:appr}. We use $n_{\rm MCMC}= 50$ steps with proposal scale $\epsilon = 0.015$. 

Comparing the stable and unstable cases, we make the following observations. For unstable systems, differences caused by small parameter variations amplify over time. As a consequence, the outputs $m_0^{(k+1)}$ of the Kalman filters (see \Cref{sec:KalmanFilter}) vary a lot for different parameters. Hence, their difference to the data, which occurs in the likelihood \eqref{eq:rholikenew}, varies significantly across different realizations of the unknown parameter. This leads to notably centered posterior density functions in \Cref{fig:osc2}. On the other hand, for stable systems, the differences caused by small parameter variations vanish over time, leading to less centered posterior densities, see \Cref{fig:osc1}).

\begin{figure}[h!]
\begin{subfigure}{0.48\textwidth}
    \centering
    \includegraphics[width=1.\textwidth]{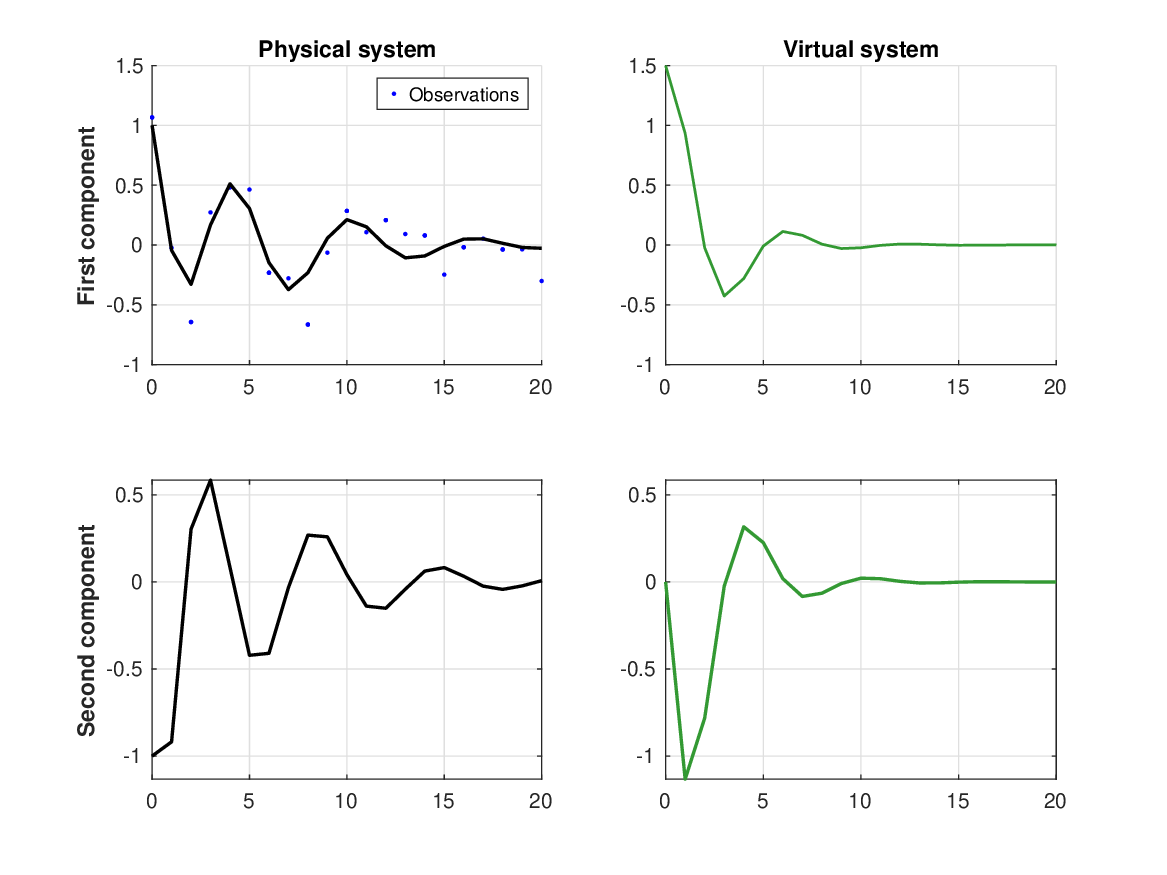}
\caption{Trajectories of the coupled physical-virtual system.}
\end{subfigure}\hfill
\begin{subfigure}{0.48\textwidth}
    \centering
    \includegraphics[width=1.\textwidth]{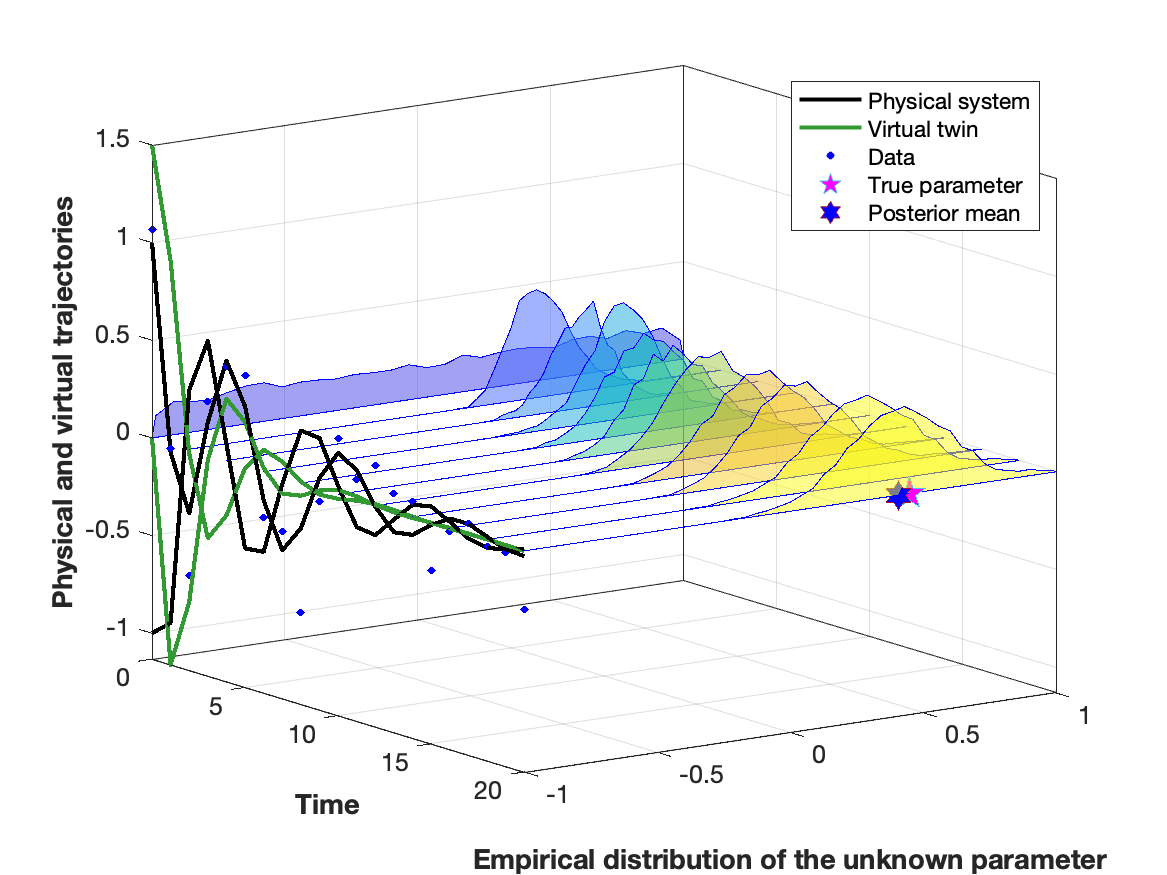}
\caption{Online estimation of the empirical densities of the uncertain parameter.}
\end{subfigure}\caption{The true parameter yields a stable system.}\label{fig:osc1}
\end{figure}

\begin{figure}[h!]
\begin{subfigure}{0.48\textwidth}
    \centering
    \includegraphics[width=1.\textwidth]{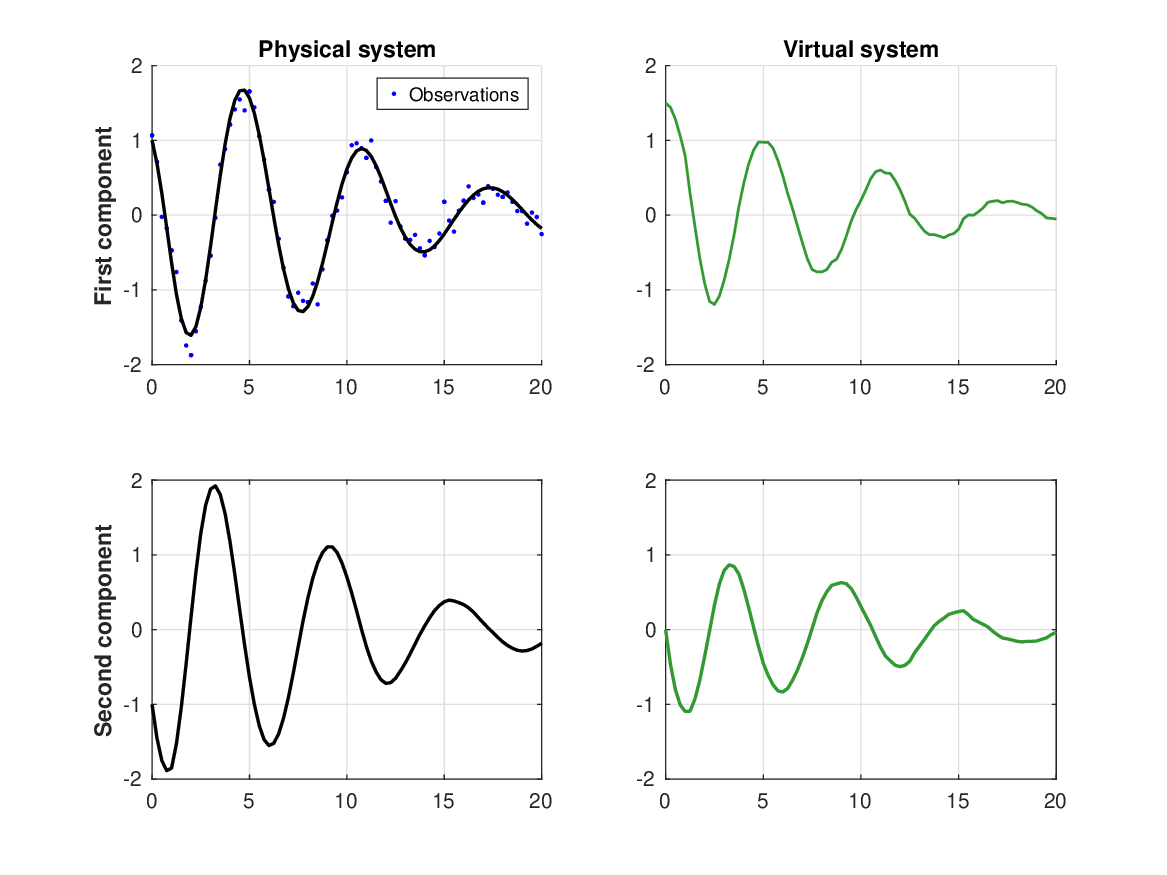}
\caption{Trajectories of the coupled physical-virtual system.}
\end{subfigure}\hfill
\begin{subfigure}{0.48\textwidth}
    \centering
    \includegraphics[width=1.\textwidth]{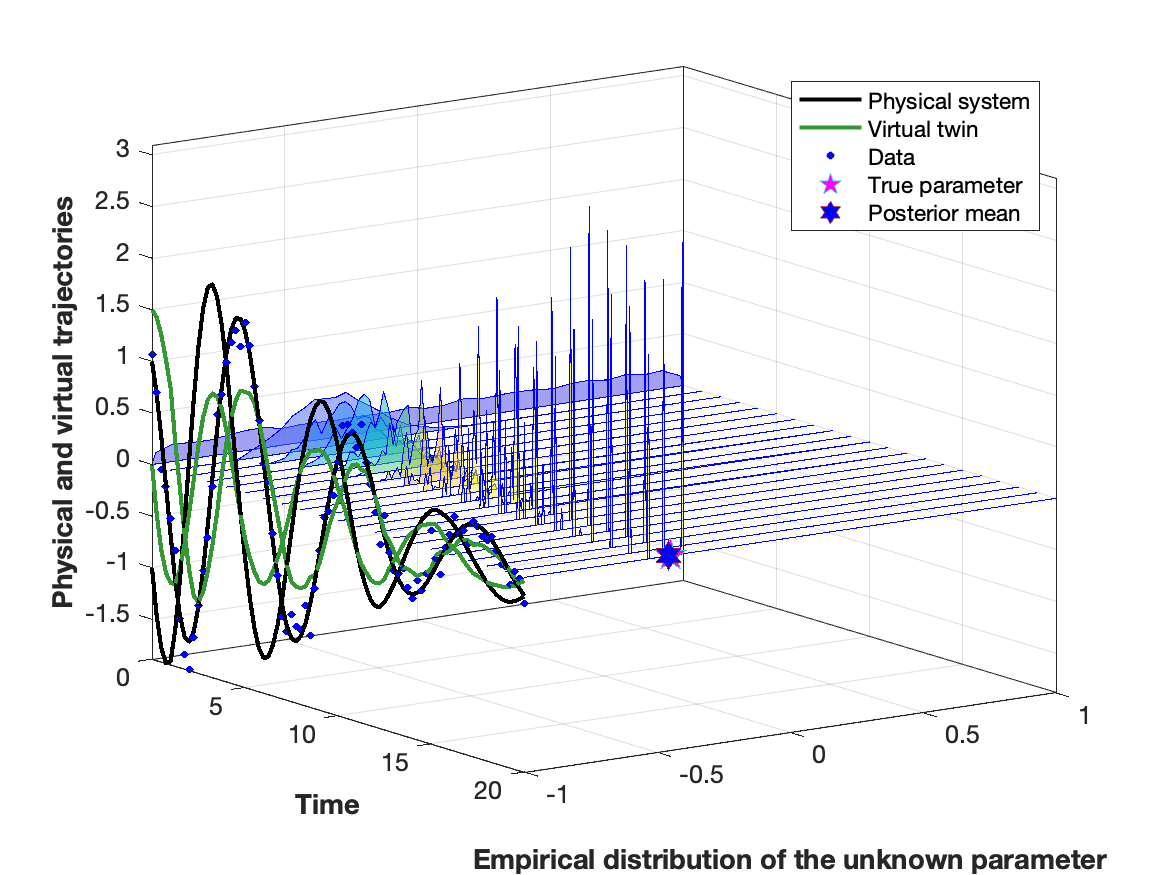}
\caption{Online estimation of the empirical densities of the uncertain parameter.}
\end{subfigure}\caption{The true parameter yields an unstable system.}\label{fig:osc2}
\end{figure}

In further numerical experiments, which are not displayed in this manuscript, we discovered that the update frequency of \(\widehat{\sigma}_k\) is crucial for the parameter estimation quality and stabilization. For instance, decreasing the update frequency, estimating at times $4,8,12,16,20$ results in failure of stabilization in the unstable case. Compared to the influence of the update frequency, the noise level and the number of observation have a smaller effect on the results.

\subsection{Spring-damper-system}\label{sec:springdamper} We next consider a finite-dimensional system which is both stabilizable and detectable, but neither controllable nor observable, see \Cref{fig:tikz} (right).

Let~$(x_1(t),x_2(t))\in[0,L]\times[0,L]$ be the position of the pair of masses~$(m_1,m_2)$, let~$(x_{1,0},x_{2,0})$ be an equilibrium position at rest, and define the relative position~$q(t)=(q_1(t),q_2(t))=(x_1(t)-x_{1,0},x_2(t)-x_{2,0})$. For the dynamics, we have the model (for small~$q$)
\begin{align}
    m_1\ddot q_1&=-c_1\dot q_1-c(\dot q_1-\dot q_2)-k(q_1-q_2)+u,\\
    m_2\ddot q_2&=-c_2\dot q_2-c(\dot q_2-\dot q_1)-k(q_2-q_1).
\end{align}
The input control forcing~$u$ acts on the mass~$m_1$ and we measure the (relative) position~$q_2$ of mass~$m_2$. Hereafter we consider the case $m_1=m_2=c_1=c=c_2=1$ with the uncertain parameter~$\sigma=k>0$.
Denoting~$Q=(q_1,q_2,\dot q_1,\dot q_2)$, we write the system as
\begin{subequations}\label{sys-2masses-sig}
\begin{align}
   \dot Q&=A_\sigma Q+Bu,\\
   \text{with}\quad A_\sigma&=\begin{bmatrix}
       0&0&1&0\\
       0&0&0&1\\
       -\sigma&\sigma&-2&1\\
       \sigma&-\sigma&1&-2
   \end{bmatrix} ,\quad
   B=\begin{bmatrix}
       0\\0\\1\\0
   \end{bmatrix},\quad
     C=\begin{bmatrix}
       0&1&0&0
   \end{bmatrix}.
\end{align}
\end{subequations}

Though the physical model above is considered with~$\sigma>0$, here we allow $\sigma$ to vary in all of $\mathbb{R}$. This is motivated by numerical tests which include the challenging case where the dynamics is exponentially unstable, if~$\sigma<0$.

It can be shown that the system~\eqref{sys-2masses-sig} is stabilizable for all~$\sigma\in\bbR$ and detectable for all~$\sigma\ne0$. Further, in case~$\sigma=1$ the system is neither controllable nor observable.

The numerical results displayed in~\Cref{fig:sds1} and~\Cref{fig:sds2} are obtained using $N = 5000$ particles initialized from a uniform prior distribution on the interval $[-1.5,1.5]$. The true parameter is $\sigma = \tfrac{\pi}{7}\approx 0.4488$ in the stable case and $\sigma = -0.1$ in the unstable case. In the stable case, the estimate $\widehat{\sigma}$ is updated at times $4,8,12,\ldots,60$ with observations at times $0.25,0.5,0.75,\ldots,60$, while in the unstable case, $\widehat{\sigma}$ is updated more frequently at times $0.5,1,1.5,\ldots,60$ with more observations at times $1/16,2/16,3/16,\ldots,60$. In both cases, we initialize with $y_0 = \begin{bmatrix} 1 & -0.5 & -1 & -0.2\end{bmatrix}^\top$ and $\widehat{y}_0 =\begin{bmatrix} -1.1 & 0 & 0 & 0 \end{bmatrix}^\top $. Furthermore, we assume that the distribution of the initial condition is the Gaussian $\mathcal{N}\Big(\begin{bmatrix} 1.5 & -1.1 & -0.8 & -0.18\end{bmatrix}^\top, {\rm Id_4}\Big)$ and that the distribution of the noise is $\mathcal{N}(0,0.15\|y_0\|)$. The matrices \(\mathcal{A}_{\sigma},\clB_{\sigma}\) are computed using the matrix exponential as $\clA_\sigma = e^{A_\sigma \Delta t}$ and $\clB_{\sigma} = \int_0^{\Delta t} e^{A_\sigma(\Delta t-\tau)}\,\mathrm d\tau\, B$, whereas the matrices \(\clA_{\widehat{\sigma}},\clB_{\widehat{\sigma}}\) are computed using a Crank--Nicolson discretization as in \Cref{re:appr}. We use $n_{\rm MCMC}= 10$ steps with proposal scale $\epsilon = 0.1$.

We observe that the system is stabilized in both the stable as well as the unstable case. In the unstable case, the posterior densities are more centered as compared to the stable case. Furthermore, we observe that the numerical stabilization of the unstable system fails if we update the estimate less frequently or have fewer observations.

\begin{figure}[h!]
\begin{subfigure}{0.48\textwidth}
    \centering
    \includegraphics[width=1.\textwidth]{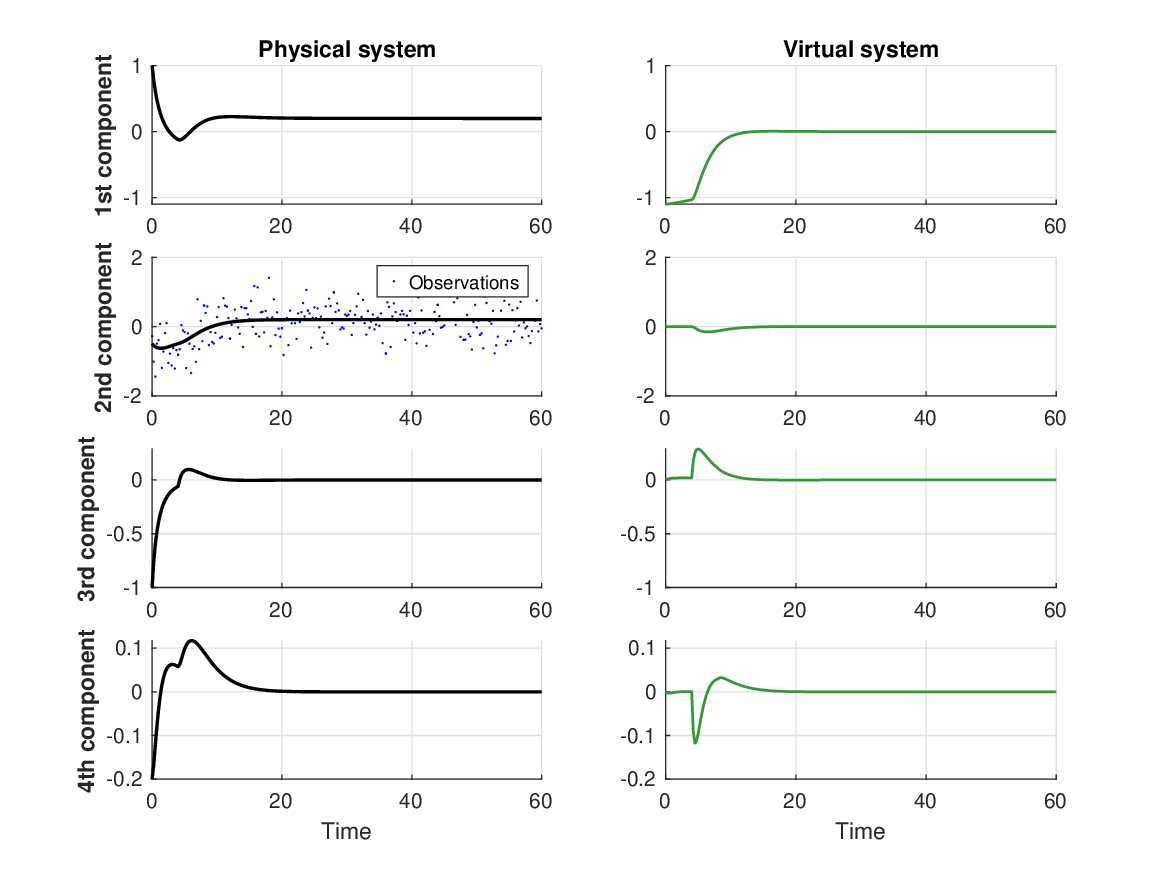}
\caption{Trajectories of the coupled digital-physical system.}
\end{subfigure}\hfill
\begin{subfigure}{0.48\textwidth}
    \centering
    \includegraphics[width=1.\textwidth]{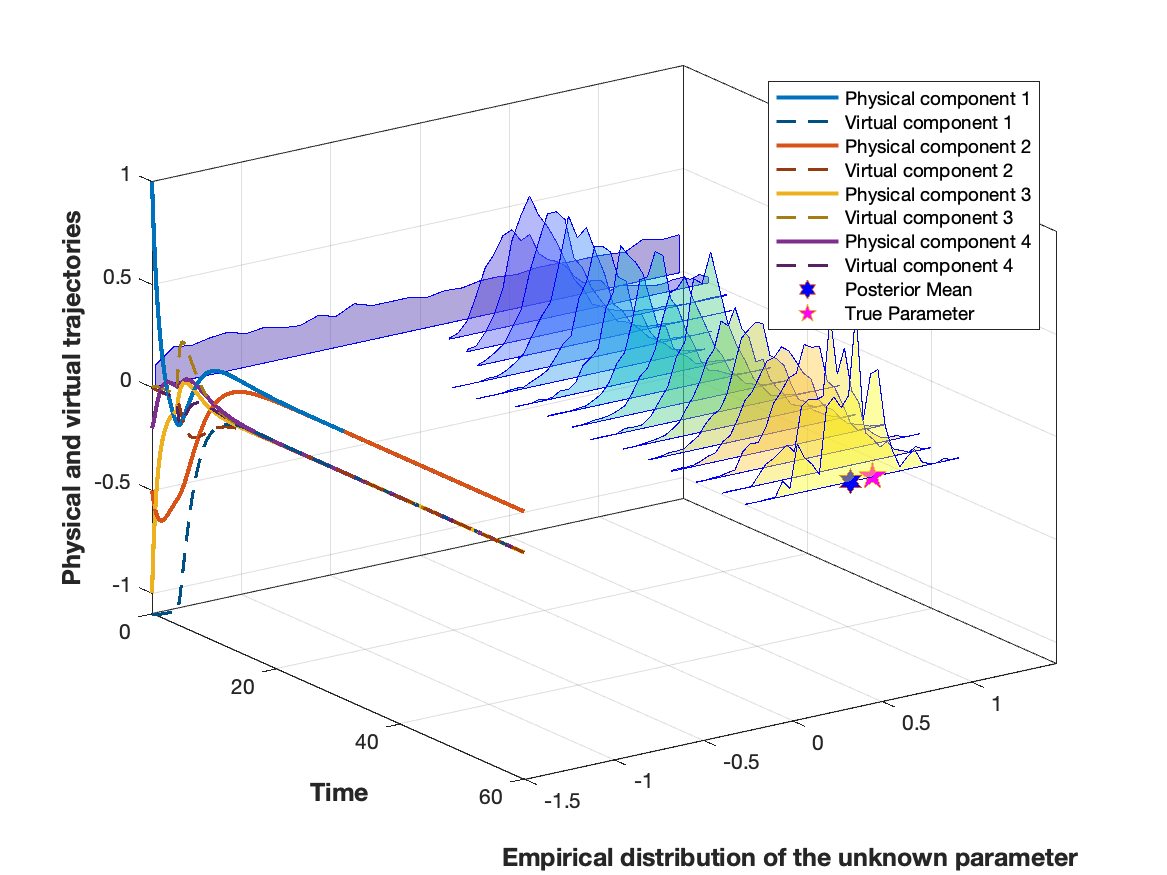}
\caption{Online estimation of the empirical densities of the uncertain parameter.}
\end{subfigure}\caption{The true parameter yields a stable system.}\label{fig:sds1}
\end{figure}

\begin{figure}[h!]
\begin{subfigure}{0.48\textwidth}
    \centering
    \includegraphics[width=1.\textwidth]{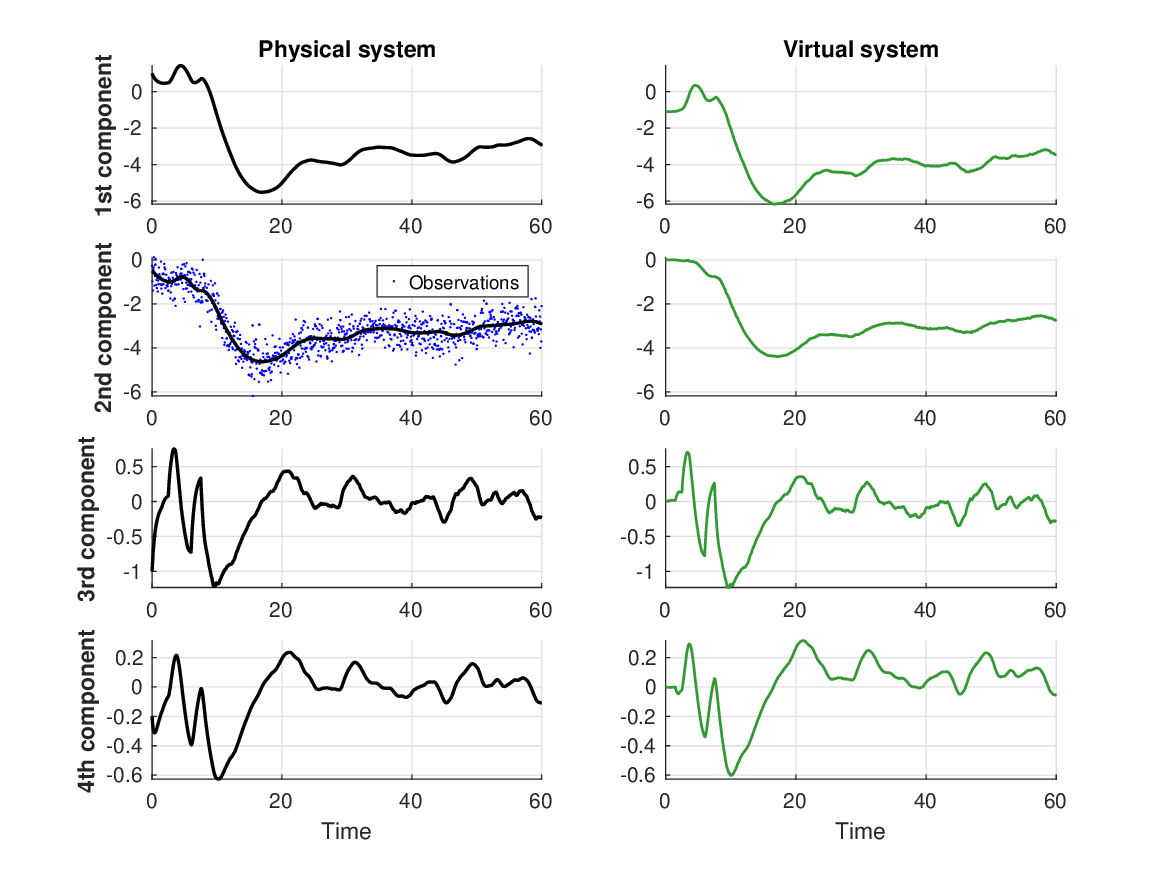}
\caption{Trajectories of the coupled digital-physical system.}
\end{subfigure}\hfill
\begin{subfigure}{0.48\textwidth}
    \centering
    \includegraphics[width=1.\textwidth]{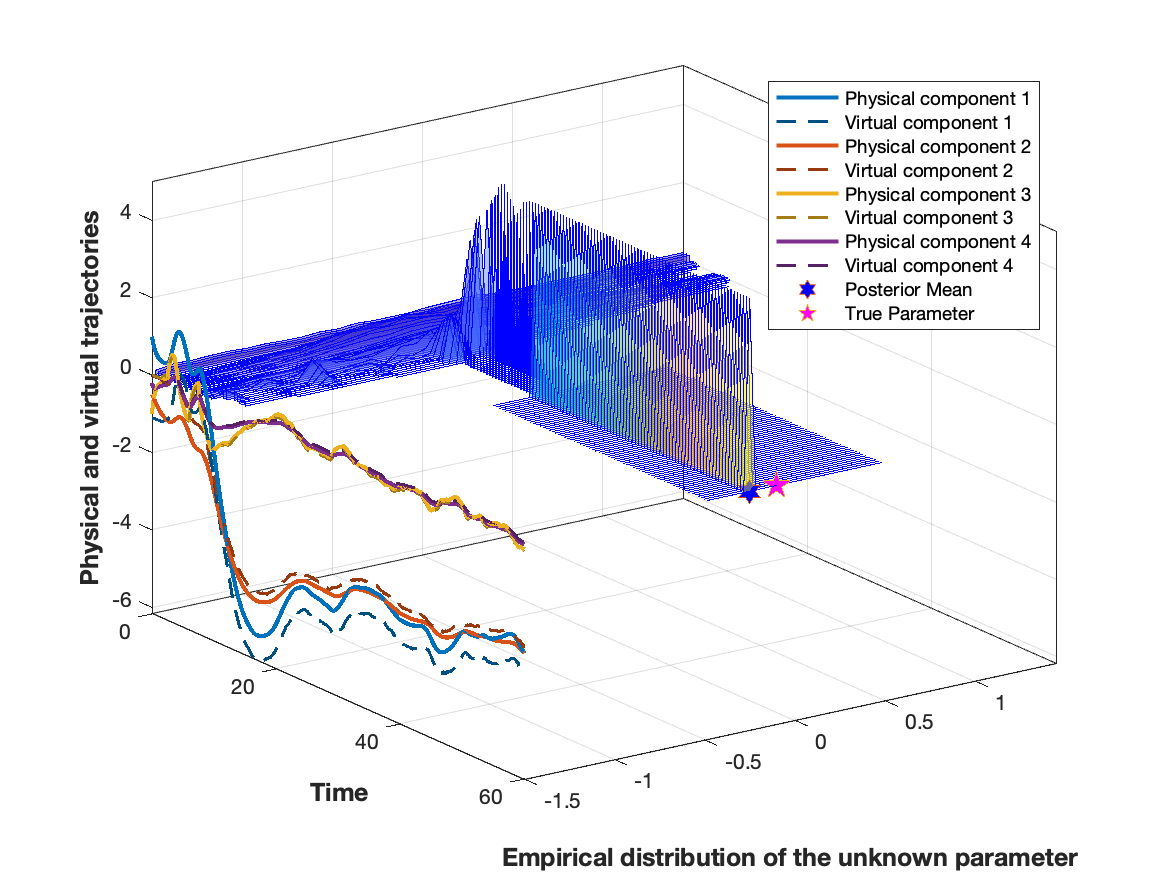}
\caption{Online estimation of the empirical densities of the uncertain parameter.}
\end{subfigure}\caption{The true parameter yields an unstable system.}\label{fig:sds2}
\end{figure}


\subsection{Diffusion-reaction equation}\label{sec:parabolic}
We consider a finite element semi-discretization of a parameterized diffusion-reaction equation.

Let~$D = (0,1)$ and, for every $\sigma$ let~$y_\sigma$ solve
\begin{subequations}\label{sys-diffreac-sig}
\begin{align}
    \tfrac{\partial}{\partial t} y_\sigma(t,\xi) - \nabla\cdot(a_\sigma(\xi) \nabla y_\sigma(t,\xi)) + c\,y_\sigma(t,\xi) &= \sum_{i=1}^{N_{\rm a}} u_i(t) \mathbf 1_{O_i}, && (t,\xi)\in(0,T]\times D,\\
    \frac{\partial y_\sigma(t,\xi)}{\partial\mathrm n} &= 0, && (t,\xi)\in[0,T]\times\partial D,\\
    y_\sigma(t,\xi) &= y_0, && (t,\xi)\in\{t=0\}\times D,
\end{align}
\end{subequations}
where~$\mathbf 1_{O_i}$ denotes the support of the~$i$-th actuator on the open set~$O_i\subset D$, for~$1\le i\le N_{\rm a}$, and~$c\in\bbR$ is a constant reaction coefficient. The uncertain, spatially varying diffusion coefficient is parameterized log-linearly by~$\sigma = [\sigma_1,\dots,\sigma_p]\in \Sigma_p \subset \bbR^p$ via
\begin{align*}
    a_\sigma(\xi) = \exp\Big(\sum_{j=1}^p \sigma_j\psi_j(\xi)\Big),
\end{align*}
with basis functions
\begin{align*}
    \psi_{2j-1}(\xi) = (2j-1)^{-\vartheta}\cos(j\pi\xi), \qquad
    \psi_{2j}(\xi) = (2j)^{-\vartheta}\sin(j\pi\xi),
\end{align*}
for $1\le j\le \lceil p/2\rceil$, so that~$\vartheta>0$ controls the decay of the parametric influence of higher modes, guaranteeing~$0 < a_{\min}\le a_\sigma(\xi)\le a_{\max}<\infty$ for~$\sigma$ in a bounded set, such as $\Sigma_p := [-10,10]$.

We discretize~\eqref{sys-diffreac-sig} in space using continuous piecewise-linear finite elements on a uniform mesh of~$D$ with~$d$ nodes and mesh width~$h=1/d$, with nodal basis~$\{\phi_\ell\}_{\ell=1}^d$. Let~$M\in\bbR^{d\times d}$ denote the finite element mass matrix, $M_{\ell,\ell'} = \int_D \phi_\ell\phi_{\ell'}\,\mathrm d\xi$, and let~$K_{a_\sigma}\in\bbR^{d\times d}$ denote the stiffness matrix associated with the weighted diffusion form, $(K_{a_\sigma})_{\ell,\ell'} = \int_D a_\sigma\,\nabla\phi_\ell\cdot\nabla\phi_{\ell'}\,\mathrm d\xi$. The homogeneous Neumann condition enters~\eqref{sys-diffreac-sig} naturally, without an explicit boundary term. Writing~$Y_\sigma(t)\in\bbR^d$ for the vector of nodal values of the finite element approximation of~$y_\sigma(t,\cdot)$, the semi-discrete system reads $M\dot Y_\sigma = -\big(K_{a_\sigma} + cM\big)Y_\sigma + \mathcal B u$, i.e.,
\begin{align*}
    \dot Y_\sigma = A_\sigma Y_\sigma + Bu, \quad A_\sigma := -M^{-1}\big(K_{a_\sigma}+cM\big),
\end{align*}
with~$B\in\bbR^{d\times N_{\rm a}}$ the discretized actuator operator, whose~$i$-th column approximates a smoothed indicator of~$O_i$. In our experiments, $N_{\rm a}=3$ actuators of width~$0.1$, uniformly distributed over~$D$ are utilized.

The output operator~$C\in\bbR^{d\times d}$ is taken as the~$M$-orthogonal projection onto the span of the leading~$N_{\rm eig}$ eigenfunctions~$\{e_j\}_{j=1}^{N_{\rm eig}}$ of the Neumann Laplacian on~$D$ (here, $e_j(\xi)=\cos((j-1)\pi\xi)$), i.e.,
\begin{align*}
    C = M E \big(E^\top M E\big)^{-1} E^\top, \qquad E = [e_1,\dots,e_{N_{\rm eig}}]\in\bbR^{d\times N_{\rm eig}},
\end{align*}
so that~$Cy$ retains the projection of~$y$ onto the~$N_{\rm eig}$ dominant spectral modes and discards the remainder. We choose~$d=64$ and~$N_{\rm eig}=16$. Furthermore, we take~$p=20$, decay exponent~$\vartheta=2$, constant reaction~$c=-1$ (so that the uncontrolled system is unstable), true parameter
\[
\sigma_{\rm true} = (-\tfrac\pi7,\ \sqrt2,\ -3,\ -1,\ 6,\ \xi_1,\dots,\xi_{15}), \qquad \xi_i \overset{\rm i.i.d.}{\sim}\mathrm{Unif}(-10,10),
\]
initial condition~$y_0(\xi) = 2\sin(2\pi\xi)$, and prior~$\sigma\sim\clN(\sigma_{\rm true}, 4\cdot\mathrm{Id}_p)$. The distribution of the initial condition is $\mathcal{N}(1.25\cdot y_0, {\rm Id_d})$.
The observation noise level is~$\eta\sim\mathcal N(0,\,0.0075\cdot\|y_0\|)$. The matrices \(\mathcal{A}_{\sigma},\clB_{\sigma}\) are computed using the matrix exponential as $\clA_\sigma = e^{A_\sigma \Delta t}$ and $\clB_{\sigma} = \int_0^{\Delta t} e^{A_\sigma(\Delta t-\tau)}\,\mathrm d\tau\, B$, whereas the matrices \(\clA_{\widehat{\sigma}},\clB_{\widehat{\sigma}}\) are computed using a Crank--Nicolson discretization as in \Cref{re:appr}. We use $n_{\rm MCMC}= 10$ steps with proposal scale $\epsilon = 0.5$. The parameter estimate is updated at times $0.05, 0.1, 0.15, 0.2$ with observations at times $1/100,2/100,3/100,\ldots,0.2$. 

We observe that the physical system is stabilized (see~\Cref{fig:heatPS}), while the virtual system tracks the physical one (compare~\Cref{fig:heatPS} with~\Cref{fig:heat_VS}) and the diffusion coefficient is estimated simultaneously (see~\Cref{fig:heat_coef}).

\begin{figure}[htbp]
    \centering

    \begin{subfigure}[b]{0.48\textwidth}
        \centering
        \includegraphics[width=\textwidth]{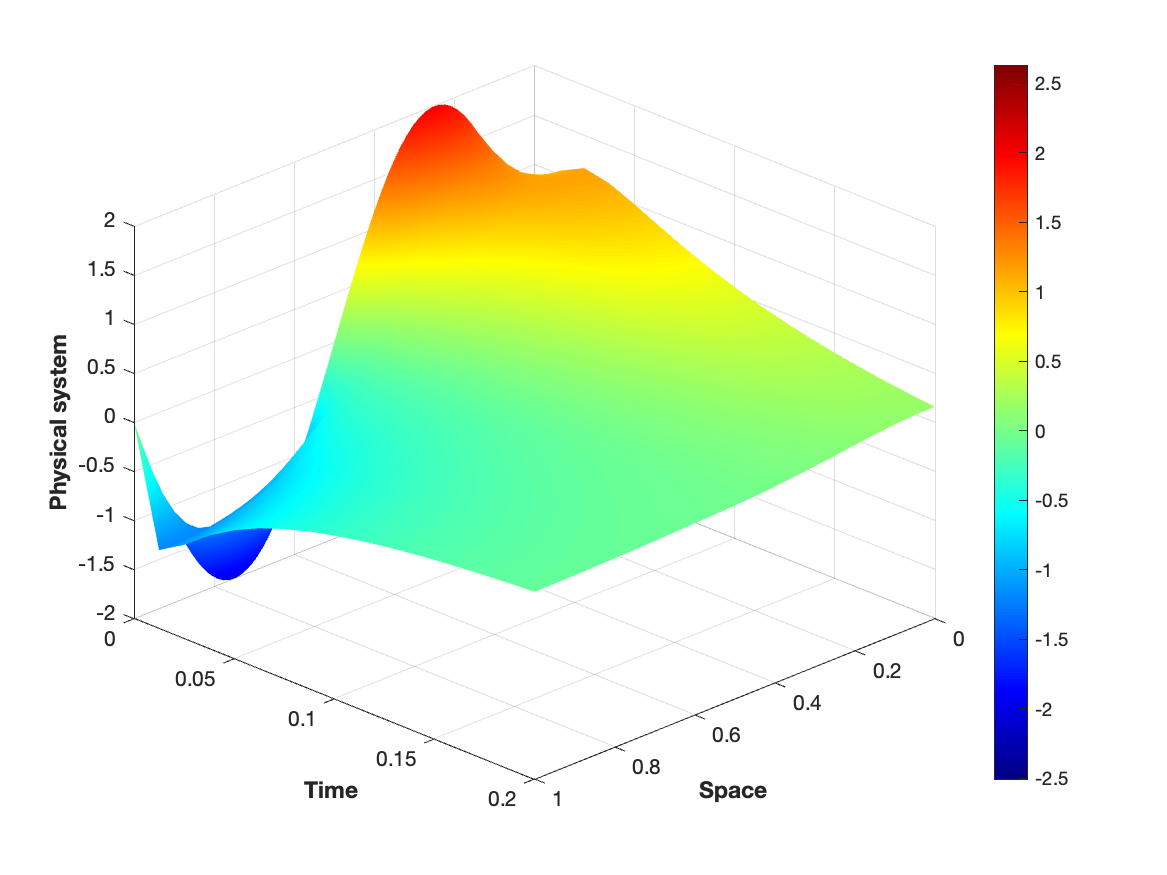}
        \caption{Physical system}
        \label{fig:heatPS}
    \end{subfigure}
    \hfill
    \begin{subfigure}[b]{0.48\textwidth}
        \centering
        \includegraphics[width=\textwidth]{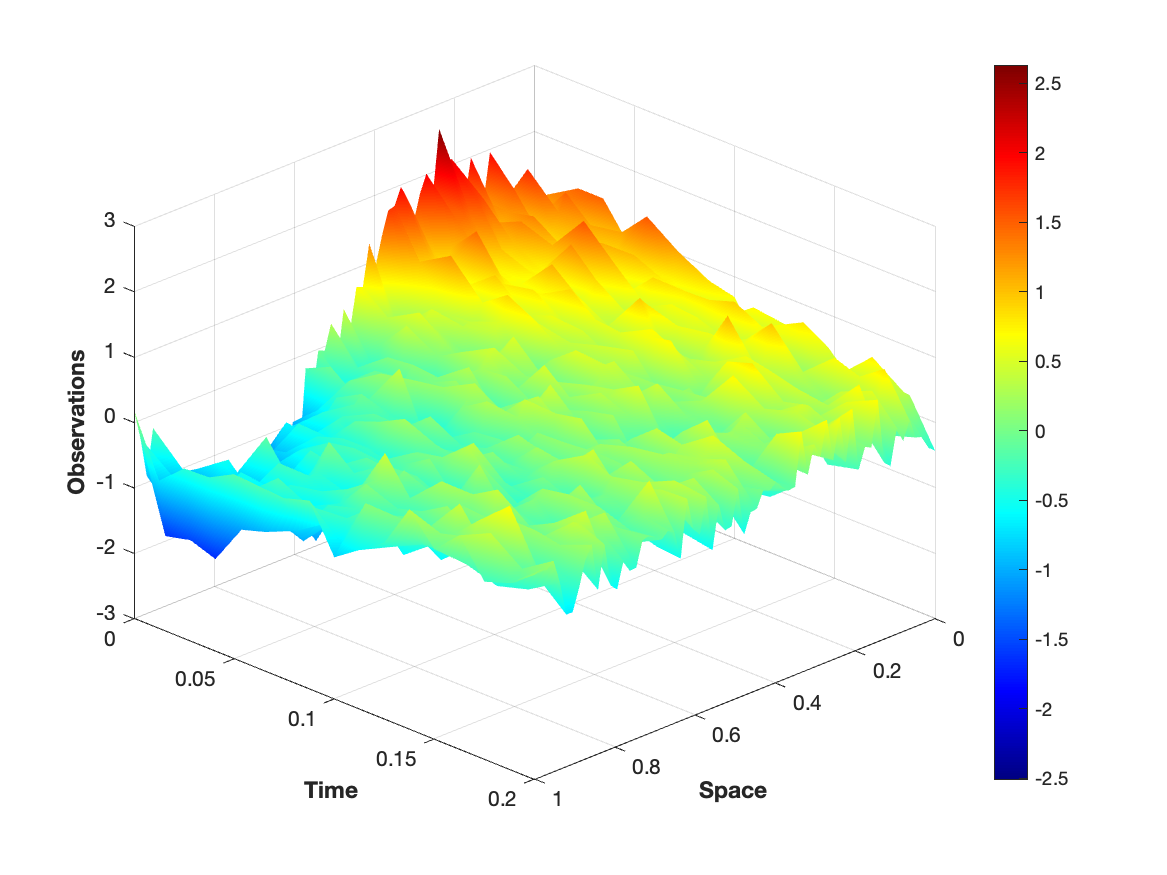}
        \caption{Observations}
        \label{fig:heat_obs}
    \end{subfigure}

    \vspace{0.5cm}

    \begin{subfigure}[b]{0.48\textwidth}
        \centering
        \includegraphics[width=\textwidth]{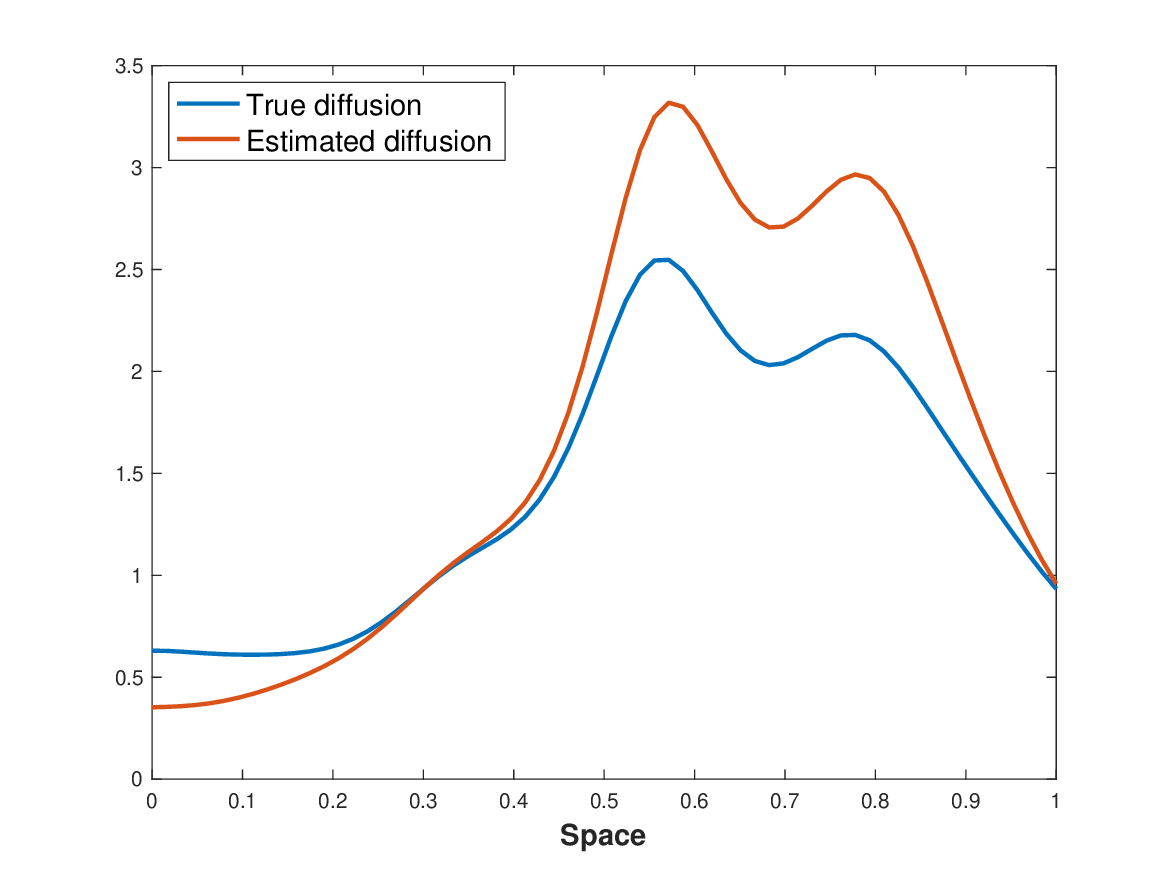}
        \caption{Diffusion coefficient at $t = 0.2$}
        \label{fig:heat_coef}
    \end{subfigure}
    \hfill
    \begin{subfigure}[b]{0.48\textwidth}
        \centering
        \includegraphics[width=\textwidth]{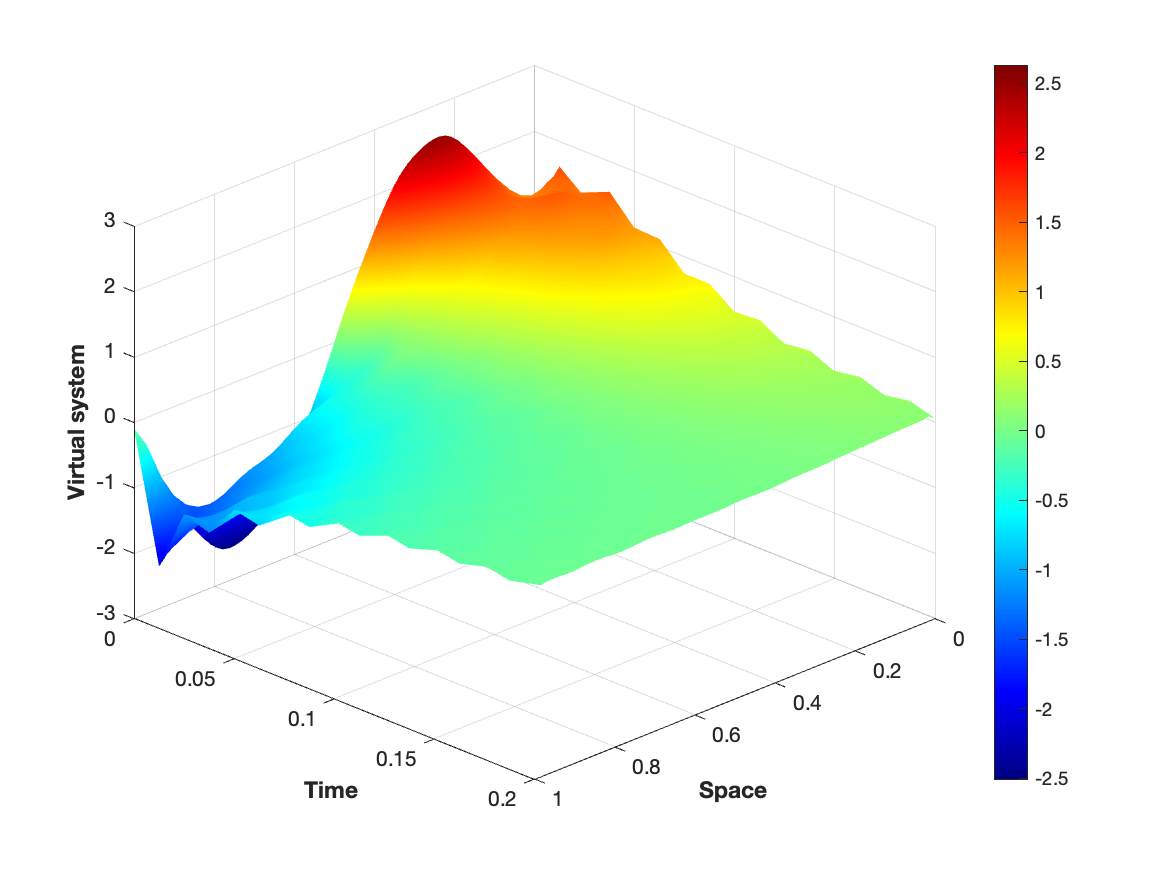}
        \caption{Virtual system}
        \label{fig:heat_VS}
    \end{subfigure}

    \caption{Results for the semi-discretized diffusion-reaction equation.}
    \label{fig:heat_results}
\end{figure}

\section{Outlook}\label{sec:outlook}

We provided a digital twin framework involving a virtual and a physical twin and a bidirectional coupling between them. As a prototypical design objective we chose the stabilization of the physical twin.  Mathematically, this involves simultaneous Bayesian inference, (optimal) feedback control, and state estimation. A significant difficulty arises due to the necessity of parameter updating, which consistently affects and changes  the optimal feedback structure of the physical twin and the system operator of the  virtual twin. 
 
We identify several open challenges for possible further research. A convergence analysis of  the parameter update procedure detailed in  \Cref{sec:param} is a challenging future goal. Similarly, for the estimation procedure, a finite set of data is utilized for each fixed control strategy. An asymptotic analysis with respect to the cardinality of the data set could be of interest, in the presence of updated controls. 
An efficient and constructive strategy for sensor placement, in relation to dynamical systems properties, for example, as well as the parameter update frequency, deserves deep investigation.

\section*{Acknowledgements}
Part of this research was stimulated by the conducive atmosphere at the Institute for Mathematical and Statistical Innovation (IMSI), during workshops within the special semester on Digital Twins in the fall of 2025.\\
This work was initiated when S.R. was with RICAM, Austrian Academy of Sciences,
part of his work is funded by the FCT -- Portuguese Foundation for Science and Technology, I.P., under the scope of the projects UID/00297/2025~(\href{https://doi.org/10.54499/UID/00297/2025}{doi: 10.54499/UID/00297/2025}) and UID/PRR/00297/2025~(\href{https://doi.org/10.54499/UID/PRR/00297/2025}{doi: 10.54499/UID/PRR/00297/2025}) (Center for Mathematics and Applications -- NOVA Math).

\bibliographystyle{abbrvurl}
\bibliography{references}

\end{document}

%% file: Mathcommands.tex
\newcommand{\rest}{\left.\kern-2\nulldelimiterspace\right|_}

\newcommand{\dnorm}[2]{\left\|#1\right\|_{#2}}

\newcommand{\Id}{{\mathbf1}}

\newcommand{\e}{\varepsilon}

\newcommand{\clA}{{\mathcal A}}
\newcommand{\clB}{{\mathcal B}}
\newcommand{\clC}{{\mathcal C}}

\newcommand{\clI}{{\mathcal I}}

\newcommand{\clK}{{\mathcal K}}

\newcommand{\clN}{{\mathcal N}}
\newcommand{\clO}{{\mathcal O}}

\newcommand{\bbN}{{\mathbb N}}

\newcommand{\bbR}{{\mathbb R}}

\newcommand{\rmd}{{\mathrm d}}

\definecolor{DarkBlue}{rgb}{0,0.08,0.45}
\definecolor{DarkRed}{rgb}{.65,0,0}
\definecolor{applegreen}{rgb}{0.55, 0.71, 0.0}

\newcounter{mymac@matlab}
\newcommand{\matlab}{MATLAB%
   \ifnum\value{mymac@matlab}<1%
   \textregistered%
   \setcounter{mymac@matlab}{1}%
   \fi%
  }


%% file: figure1.tex
\bigskip
\begin{minipage}[t]{0.43\textwidth}
    \begin{center}
        \begin{tikzpicture}[scale=0.5, every node/.style={scale=.55}, font=\sffamily, align=center]
            \node[draw, very thick, fill=blue!10, rounded corners=5pt, font=\huge,
            minimum width=4cm, minimum height=1.4cm, align=center] (phys)
            {\textbf{Physical System}};
        \end{tikzpicture}
    \end{center}
    \vskip 15pt
    \textbf{Real-world} physical phenomenon evolves as
        \begin{equation*}
            \dot y(t)=A_{\textcolor{red!70!black}{\sigma}} y(t) +B\textcolor{green!50!black}{u(t)},
        \end{equation*}
    for time~$t>0$, with \textbf{unknown initial state}~$y(0) = y_0 \sim \mathcal{N}(m_0,\mathcal{C}_0)$
    and \textcolor{red!70!black}{unknown parameter $\textcolor{red!70!black}{\sigma}$.}
\end{minipage}
\hskip 10pt%
\begin{minipage}[t]{0.43\textwidth}
    \begin{center}
        \begin{tikzpicture}[remember picture,scale=0.5, every node/.style={scale=.55}, font=\sffamily, align=center]
            \node[draw, very thick, blue!70!black, rounded corners=5pt, font=\huge,
            minimum width=4cm, minimum height=1.4cm, align=center] (meas)
            {\textbf{Measurements}};

            \coordinate (meas-east) at (meas.east);
        \end{tikzpicture}
    \end{center}
    \vskip 15pt
    Take~$s$ discrete perturbed observations of the state as
        \begin{equation*}
        \textcolor{blue}{z(t_j)}=Cy(t_j) + \eta_j,\qquad 1\le j\le s,
    \end{equation*}
    where~$\eta_j\sim \mathcal{N}(0,\Gamma_j)$ are independent.
\end{minipage}

\begin{tikzpicture}[very thick, blue!70!black]
  \draw[->]
    ([yshift=5pt]phys.east)
    .. controls +(1.1,0.) and +(-1.1,0.) ..
    node[above, yshift=2pt]{}
    ([yshift=5pt]meas.west);
\end{tikzpicture}


\bigskip\bigskip\bigskip\bigskip
\hspace{-3cm}
    \begin{tikzpicture}[scale=0.5, every node/.style={scale=.55}, font=\sffamily, align=center]
    \node[draw, very thick, fill=black!10, rounded corners=5pt, font=\huge,
    minimum width=4cm, minimum height=1.4cm, align=center] (goal)
    {\textbf{Goal:} stabilize the physical system\\$y(t) \to 0$};
    \end{tikzpicture}
\bigskip\bigskip\bigskip

\begin{tikzpicture}[remember picture,overlay]
    \coordinate (midpoint) at ($(meas.east)+(-24mm,-30mm)$);  \node at (midpoint) [xshift=18mm,yshift=-14mm,inner sep=0pt] {\begin{tikzpicture}
  \begin{axis}[
    width=4.2cm,
    height=3cm,
    hide axis,              
    clip=false,             
    samples=200,
    every axis plot/.append style={thick},
  ]
    \addplot[red!70!black, domain=0:10]
      {exp(-(x-3)^2/2)/sqrt(2*pi)};

    \addplot[magenta!100!black, domain=0.01:10]
      {1/(x*0.6*sqrt(2*pi))*exp(-((ln(x)-1)^2)/(2*0.6^2))};
  \end{axis}
\end{tikzpicture}
};
  \node at ($(midpoint)+( +20mm,-18mm)$) [font=\normalsize] {estimate unknown parameter $\textcolor{magenta}{\widehat{\sigma}}$};
\end{tikzpicture}


\noindent
\begin{minipage}[t]{0.43\textwidth}
    \begin{center}
        \begin{tikzpicture}[scale=0.5, every node/.style={scale=.55}, font=\sffamily, align=center]
            \node[draw, very thick, green!50!black, rounded corners=5pt, font=\huge,
            minimum width=4cm, minimum height=1.4cm, align=center] (control)
            {\textbf{Control input}};
        \end{tikzpicture}
    \end{center}
    \vskip 15pt
    The \textbf{control input} for the physical system
        \begin{equation*}
            \textcolor{green!50!black}{u(t)}=K_{\textcolor{magenta}{\widehat{\sigma}}} \widehat{y}(t),
        \end{equation*}
    is designed based on the virtual system.
\end{minipage}
\hskip 5pt
\begin{minipage}[t]{0.43\textwidth}
    \begin{center}
        \begin{tikzpicture}[remember picture,scale=0.5, every node/.style={scale=.55}, font=\sffamily, align=center]
            \node[draw, very thick, font=\huge, fill=green!10, rounded corners=5pt,
            minimum width=4.6cm, minimum height=1.4cm, align=center,
            right=3.85cm of control] (dig) {\textbf{Virtual System}};
              
            \coordinate (dig-east)  at (dig.east);
        \end{tikzpicture}
    \end{center}
    \vskip 15pt
    \textbf{Mirrors} the \textbf{physics} and continuously corrects itself 
    \begin{equation*}
        \dot{\widehat y}(t) = A_{\textcolor{magenta}{\widehat{\sigma}}} \widehat y(t)
        + B K_{\textcolor{magenta}{\widehat{\sigma}}} \widehat y(t)
        + L_{\textcolor{magenta}{\widehat{\sigma}}} (C\widehat y(t) - \textcolor{blue}{z(t)}),
    \end{equation*}
    where~$K_{\textcolor{magenta}{\widehat{\sigma}}}$ and
    $L_{\textcolor{magenta}{\widehat{\sigma}}}$ are constructed with the estimate
    $\textcolor{magenta}{\widehat{\sigma}}$.
\end{minipage}


\begin{tikzpicture}[very thick, green!70!black]
  \draw[<-]
    ([yshift=0pt]control.east)
    .. controls +(1.1,-0.) and +(-1.1,-0.) ..
    ([yshift=0pt]dig.west);
\end{tikzpicture}

\begin{tikzpicture}[very thick, black!90]
  \draw[->]
    ([yshift=0pt]control.west)
    .. controls +(-3.1,0.) and +(-3.3,0.) ..
    ([xshift=-0pt,yshift=5pt]phys.west);
\end{tikzpicture}

\begin{tikzpicture}[very thick, magenta!100!black]
  \draw[->]
    ([yshift=5pt]meas.east)
    .. controls +(3.5,0.) and +(2.5,0.) ..
    ([xshift=-0pt,yshift=0pt]dig.east);
\end{tikzpicture}

\vspace{-1cm}